\documentclass[12pt,twoside,reqno]{amsart}
\usepackage{amsmath}
\usepackage{amsfonts}
\usepackage{amssymb}
\usepackage{color}
\usepackage{mathrsfs}
\usepackage{cite}
\usepackage{geometry}
\usepackage{marginnote}
\usepackage{todonotes}
\allowdisplaybreaks
\newcommand{\ml}{\mathcal{L}}

\newtheorem{theorem}{Theorem}[section]
\newtheorem{corollary}[theorem]{Corollary}
\newtheorem{lemma}[theorem]{Lemma}
\newtheorem{proposition}[theorem]{Proposition}

\newtheorem{remark}[theorem]{Remark}

\allowdisplaybreaks
\numberwithin{equation}{section}
\begin{document}
\title{A nonlocal Gray-Scott model:\\ well-posedness and diffusive limit} 
\thanks{}
%
\author{Philippe Lauren\c{c}ot}
\address{Laboratoire de Math\'ematiques (LAMA) UMR~5127, Universit\'e Savoie Mont Blanc, CNRS\\	F--73000 Chamb\'ery, France}
\email{philippe.laurencot@univ-smb.fr}
\author{Christoph Walker}
\address{Leibniz Universit\"at Hannover\\ Institut f\" ur Angewandte Mathematik \\ Welfengarten 1 \\ D--30167 Hannover\\ Germany}
\email{walker@ifam.uni-hannover.de}
\date{\today}

\begin{abstract}
Well-posedness in $L_\infty$ of the nonlocal Gray-Scott model is studied for integrable kernels, along with the stability of the semi-trivial spatially homogeneous steady state. In addition, it is shown that the solutions to the nonlocal Gray-Scott system converge to those to the classical Gray-Scott system in the diffusive limit.
\end{abstract}

\maketitle

%
%
\pagestyle{myheadings}
\markboth{\sc{Ph.~Lauren\c cot \& Ch. Walker}}{\sc{A nonlocal Gray-Scott model}}

\section{Introduction}\label{sec1}

Let $\Omega$ be a bounded domain of $\mathbb{R}^n$, $n\ge 1$, and $(d_1,d_2,f,\kappa)\in (0,\infty)^4$. The Gray-Scott model 
\begin{equation}
	\begin{split}
		\partial_t u & = d_1 \Delta u - u v^2 + f (1-u) \;\;\text{ in }\;\; (0,\infty)\times\Omega\,, \\
		\partial_t v & = d_2 \Delta v + u v^2 - (f + \kappa) v \;\;\text{ in }\;\; (0,\infty)\times\Omega\,, 
	\end{split}\label{GS}
\end{equation} 
is a mathematical model for the autocatalytic chemical reaction $U+2V\to 3V$, where $u$ and~$v$ denote the concentrations of the chemical species $U$ and~$V$, respectively. Despite its simple polynomial mass-action structure, it is well-known to exhibit a rich dynamics, generating a variety of patterns according to the values of the parameters $(d_1,d_2,f,\kappa)$, including spots, stripes, and labyrinthine patterns, see~\cite{Cas2017, ChWa2011, DKZ1997, GZK2018, HPT2000, KWW2006, MoKa2004, Pea1993, PeWa2009} for instance.

A more recent trend is the study of the influence of the diffusion on pattern formation, replacing the standard Laplace operator $(d_1 \Delta,d_2 \Delta)$ by nonlocal operators such as fractional Laplacians~\cite{WSWK2019}
\begin{equation}
	\begin{split}
		\partial_t u & = d_1 (-\Delta)^{\alpha/2} u - u v^2 + f (1-u) \;\;\text{ in }\;\; (0,\infty)\times\Omega\,,  \\
		\partial_t v & = d_2 (-\Delta)^{\alpha/2} v + u v^2 - (f + \kappa) v \;\;\text{ in }\;\; (0,\infty)\times\Omega\,, 
	\end{split}\label{FLGS}
\end{equation}
with $\alpha\in (1,2)$ or convolution operators~\cite{CJW2022}
\begin{equation}
	\begin{split}
		\partial_t u & = d_1 \Gamma_{\gamma_1} u - u v^2 + f (1-u) \;\;\text{ in }\;\; (0,\infty)\times\Omega\,,  \\
		\partial_t v & = d_2 \Gamma_{\gamma_2} v + u v^2- (f + \kappa) v \;\;\text{ in }\;\; (0,\infty)\times\Omega\,,
	\end{split}\label{NLGS}
\end{equation}
where 
\begin{equation}
	\Gamma_{\gamma_\ell} z(x) := \int_\Omega \gamma_\ell(x,y) \big( z(y) - z(x)\big)\,\mathrm{d} y\,, \qquad x\in\Omega\,, \label{nlo}
\end{equation}
and $\gamma_\ell: \Omega\times\Omega\to [0,\infty)$ is a measurable function for $\ell\in\{1,2\}$. In particular, numerical simulations performed in~\cite{WSWK2019} reveal that the patterns are more localized when $\alpha<2$, while topological changes are reported in~\cite{CJW2022, WSWK2019} when the nonlocal operators in~\eqref{FLGS} or in~\eqref{NLGS} deviate sufficiently from standard diffusion. Nevertheless, the observed patterns look similar when the nonlocal operators are close to standard diffusion and one aim of this paper is to show that solutions to~\eqref{NLGS} converge to that to~\eqref{GS} for a suitable choice of the nonlocal operators $(\Gamma_{\gamma_1},\Gamma_{\gamma_2})$, see Theorem~\ref{thm1.3} below. As a preliminary step, we first establish the global well-posedness of the initial value problem associated with~\eqref{NLGS}; that is,
\begin{subequations}\label{CNLGS}
	\begin{align}
		\partial_t u & = d_1 \Gamma_{\gamma_1} u - u v^2 + f (1-u) \;\;\text{ in }\;\; (0,\infty)\times\Omega\,,  \label{uNLGS} \\
		\partial_t v & = d_2 \Gamma_{\gamma_2} v + u v^2- (f + \kappa) v \;\;\text{ in }\;\; (0,\infty)\times\Omega\,, \label{vNLGS} \\
		(u,v)(0) & = \big( u^0,v^0) \;\;\text{ in }\;\; \Omega\,,  \label{iNLGS}
	\end{align}
\end{subequations} 
the nonlocal operators being defined in~\eqref{nlo}. Throughout this paper, we set
\begin{equation*}
	X := L_\infty(\Omega)\,, \qquad X^+ := \{ z \in X\ :\ z \ge 0 \;\text{ a.e. in }\;\Omega\}\,,
\end{equation*}
and we denote the norm on $L_\infty(\Omega)$ by $\|\cdot\|_\infty = \|\cdot\|_{L_\infty(\Omega)}$.
 
\begin{theorem}[Well-posedness]\label{thm1.1}
	Assume that $\gamma_\ell: \Omega\times\Omega\to [0,\infty)$ is a measurable function satisfying 
	\begin{equation}
		\int_{\Omega} \gamma_\ell(y,x)\,\mathrm{d} y = \int_{\Omega} \gamma_\ell(x,y)\,\mathrm{d} y \le  \gamma_\infty<\infty\,, \qquad x\in\Omega\,, \label{gammaN}
	\end{equation}
	for $\ell\in\{1,2\}$ and some $\gamma_\infty\ge 1$ and consider $(u^0,v^0)\in X^+\times X^+$. Then there is a unique non-negative global solution 
	\begin{equation*}
		(u,v)\in C^1\big([0,\infty),X^+\times X^+\big)
	\end{equation*}
	to~\eqref{CNLGS} which is bounded; that is, 
	\begin{equation*}
		\sup_{t\ge 0} \big\{ \|u(t)\|_{\infty} + \|v(t)\|_{\infty} \big\} < \infty\,.
	\end{equation*}
	In fact, the mapping $(u^0,v^0)\mapsto (u,v)$ defines a global semiflow on $X^+\times X^+$.
	
	Assume further that $\gamma_\ell\in C(\bar{\Omega}\times \bar{\Omega})$ for $\ell\in\{1,2\}$. Then the same result is true when replacing $X$ by~$C(\bar{\Omega})$.
\end{theorem}

A similar result is available for the original Gray-Scott model~\eqref{GS} supplemented with homogeneous Neumann or Dirichlet boundary conditions and follows from the general theory developed in~\cite[Theorems~1-2]{HMP1987} for reaction-diffusion systems with mass balance and polynomial nonlinearities.

Theorem~\ref{thm1.1} extends~\cite[Theorem~1]{CJW2022} which is restricted to the one-dimensional case $n=1$ and only provides local well-posedness in $H^1(\Omega)\times H^1(\Omega)$. Instead of using a Galerkin approximation as in~\cite{CJW2022}, we exploit here the semilinear structure of~\eqref{CNLGS}, along with the boundedness of the nonlocal operators in~$X^2$ and the local Lipschitz continuity of the nonlinear reaction terms, still in~$X^2$, see Section~\ref{sec2}.  Notice also that the equality in condition~\eqref{gammaN} is obviously satisfied when $\gamma_\ell$ is symmetric.

\begin{remark}\label{rem1.4}
	Unlike~\eqref{GS}, there is no need to supplement~\eqref{CNLGS} with boundary conditions but the choice of the nonlocal operator~\eqref{nlo} actually includes a nonlocal version of homogeneous Neumann boundary conditions. While we refer for instance to~\cite{AVMRTM2010} for a more complete discussion on nonlocal operators and boundary conditions, we point out here that Theorem~\ref{thm1.1} is also valid for nonlocal operators corresponding to homogeneous Dirichlet boundary conditions, as shown in Section~\ref{sec5}.
\end{remark}

We next turn to the large time behaviour of solutions to~\eqref{CNLGS} and first recall that the spatially homogeneous steady states of~\eqref{NLGS} are well identified:
\begin{itemize}
	\item [(s1)] if $f<4(f+\kappa)^2$, then $(1,0)$ is the unique spatially homogeneous steady state of~\eqref{NLGS} and it is linearly stable;
	\item [(s2)] if $f= 4(f+\kappa)^2$, then~\eqref{NLGS} admits two spatially homogeneous steady states $(1,0)$ and \mbox{$(1/2,2(f+\kappa))$};
	\item[(s3)] if $f > 4(f+\kappa)^2$, then~\eqref{NLGS} admits three spatially homogeneous steady states $(1,0)$, $(u_+,v_+)$, and $(u_-,v_-)$ given by
	\begin{equation}\label{steadystate}
u_\pm := \frac{2(f+\kappa)^2}{f\pm \sqrt{f^2-4f(f+\kappa)^2}}\,,\qquad	v_\pm := \frac{f\pm \sqrt{f^2-4f(f+\kappa)^2}}{2(f+\kappa)} \,.
	\end{equation}
\end{itemize} 
We refer to~\cite{MGR2004} for additional information on these steady states in the absence of diffusion and to~\cite{Cas2017, ChWa2011, DKZ1997, HPT2000, KWW2006, MoKa2004, PeWa2009} for the existence of spatially inhomogeneous solutions to~\eqref{GS} in~$\mathbb{R}$ triggered by the availability of multiple spatially homogeneous steady states. Extending such results to~\eqref{NLGS} does not seem to be obvious and we thus restrict ourselves to the behaviour of solutions to~\eqref{CNLGS} in a neighbourhood of the semi-trivial steady state~$(1,0)$.

\begin{theorem}[Stabilization]\label{thm1.2}
	Assume that $\gamma_\ell: \Omega\times\Omega\to [0,\infty)$ is a measurable function satisfying~\eqref{gammaN} for $\ell\in \{1,2\}$. Then the semi-trivial steady state $(1,0)$ is locally asymptotically stable in~$X^2$.
	
	Moreover, if
	\begin{equation}\label{init}
		\|u^0\|_\infty\le 1+\delta \;\;\text{ and }\;\; \|v^0\|_\infty< \frac{f+\kappa}{1+\delta}
	\end{equation} 
	for some $\delta\ge 0$, then
	\begin{equation*}
		\lim_{t\to\infty} \big\{ \|u(t)-1\|_\infty + \|v(t)\|_\infty \big\} = 0
	\end{equation*}
	with
	\begin{equation*}
		\|u(t)\|_\infty\le 1+\delta\,,\qquad \|v(t)\|_\infty\le \|v^0\|_\infty\,, \qquad t\ge 0\,.
	\end{equation*}
\end{theorem}

While the first assertion in Theorem~\ref{thm1.2} is a consequence of the principle of linearized stability, the second one stems from the dissipativity properties of the operator $d_2 \Gamma_{\gamma_2} - (f+\kappa)$ in $X$, see Section~\ref{sec3}.

\bigskip

The last contribution of this paper deals, as mentioned previously, with the connection between the nonlocal Gray-Scott model~\eqref{NLGS} and the (local) Gray-Scott model~\eqref{GS}.  To study this question, we assume that the kernels $(\gamma_1,\gamma_2)$ are given by
\begin{equation}
	\gamma_1(x,y) = \gamma_2(x,y) = j^{n+2} \varphi(j(x-y))\,, \qquad (x,y)\in \Omega^2\,, \quad j\ge 1\,, \label{sc1}
\end{equation}
where $\varphi\in C_0^\infty(\mathbb{R}^n)$ is a non-negative radially symmetric function with compact support and a non-increasing radial profile; that is, $\tilde{\varphi}(r) := \varphi(x)$ for $x\in\mathbb{R}^n$ with $|x|=r$ and $r\ge 0$ is a non-increasing function on~$[0,\infty)$. Such a connection is already alluded to in~\cite[Section~4.5]{CJW2022} and the purpose of the next result is to provide a mathematical proof of this connection.

\begin{theorem}[Diffusive limit]\label{thm1.3}
	Let $n\ge 2$ and assume that $\Omega$ is a bounded domain with $C^{2+\alpha}$-smooth boundary $\partial\Omega$. Consider a non-negative radially symmetric function  $\varphi\in C_0^\infty(\mathbb{R}^n)$ with compact support and a non-increasing radial profile. For each integer $j\ge 1$, we define $\chi_{j}: \Omega\times\Omega\to [0,\infty)$ by
	\begin{equation*}
		\chi_{j}(x,y) := j^{n+2} \varphi(j(x-y))\,, \qquad (x,y)\in \Omega^2\,,
	\end{equation*}
	and, for $z\in X$, 
	\begin{equation*}
		\Gamma_{\chi_{j}} z(x) := \int_\Omega \chi_{j}(x,y) (z(y)-z(x))\,\mathrm{d}y\,, \qquad x\in\Omega\,.
	\end{equation*}
	Given $(u^0,v^0)\in X^+\times X^+$, let $(u_j,v_j)\in C^1([0,\infty),X^+\times X^+)$ be the solution to
	\begin{subequations}\label{CjNLGS}
		\begin{align}
			\partial_t u_j & = d_1 \Gamma_{\chi_j} u_j - u_j v_j^2 + f (1-u_j) \;\;\text{ in }\;\; (0,\infty)\times\Omega\,,  \label{ujNLGS} \\
			\partial_t v_j & = d_2 \Gamma_{\chi_j} v_j + u_j v_j^2- (f + \kappa) v_j \;\;\text{ in }\;\; (0,\infty)\times\Omega\,, \label{vjNLGS} \\
			(u_j,v_j)(0) & = \big( u^0,v^0) \;\;\text{ in }\;\; \Omega\,.  \label{ijNLGS}
		\end{align}
	\end{subequations} 
There are a subsequence $(j_k)_{k\ge 1}$ and non-negative functions 
\begin{equation}
	\begin{split}
		u& \in L_\infty((0,\infty)\times\Omega) \cap L_{2,\mathrm{loc}}([0,\infty),H^1(\Omega))\,, \\
		v& \in L_\infty((0,\infty),L_2(\Omega)) \cap L_{2,\mathrm{loc}}([0,\infty),H^1(\Omega))\,,
	\end{split} \label{p3}
\end{equation}	
such that
\begin{equation}
	\lim_{k\to\infty} \int_0^T \int_\Omega \left( |(u_{j_k}-u)(t,x)|^2 + |(v_{j_k}-v)(t,x)|^2 \right) \,\mathrm{d}x \mathrm{d}t  = 0 \label{p4}
\end{equation}	
for all $T>0$. Moreover, $(u,v)$ is a weak solution to~\eqref{GS} supplemented with homogeneous Neumann boundary conditions in the following sense: 
\begin{equation}
	\begin{split}
		\int_\Omega \big( u(t,x)-u^0(x) \big) \vartheta(x)\,\mathrm{d}x & + \frac{m_2 d_1}{2n} \int_0^t \int_\Omega 	\nabla 	u(s,x)\cdot \nabla\vartheta(x)\,\mathrm{d}x\mathrm{d}s \\
		& \qquad = \int_0^t \int_\Omega \big( f(1-u) - uv^2 \big)(s,x)\,\mathrm{d}x\mathrm{d}s
	\end{split} \label{p5}
\end{equation}
and
\begin{equation}
	\begin{split}
		\int_\Omega \big( v(t,x)-v^0(x) \big) \vartheta(x)\,\mathrm{d}x & + \frac{m_2 d_2}{2n} \int_0^t \int_\Omega 	\nabla 	v(s,x)\cdot \nabla\vartheta(x)\,\mathrm{d}x\mathrm{d}s \\
		& \qquad = \int_0^t \int_\Omega \big(uv^2 - (f+\kappa) v \big)(s,x)\,\mathrm{d}x\mathrm{d}s
	\end{split} \label{p6}
\end{equation}
for all $\vartheta\in W^1_{n+1}(\Omega)$ and $t>0$, where
\begin{equation}
	m_2 := \int_{\mathbb{R}^n} |x|^2 \varphi(x)\,\mathrm{d}x>0\,. \label{m2}
\end{equation}.
\end{theorem}

In recent years, several works have been devoted to the proof of the convergence of nonlocal equations or systems to their diffusive limits and, among others, we refer to \cite{AVMRTM2010} for $p$-Laplacian equation, to \cite{AbHu2023, CES2023, DRST2020, DST2021, ElSk2023, MRST2019} for degenerate and non-degenerate Cahn-Hilliard equations and systems, to \cite{AbTe2022} for the Navier-Stokes/Cahn-Hilliard system, and to~\cite{Mou2020} for cross-diffusion systems with triangular structure. As in \cite{CES2023, ElSk2023}, the main difficulty we face here is to establish the strong compactness of the sequence $(u_j,v_j)_{j\ge 1}$ as the nonlocal operators do not provide estimates in $H^1(\Omega)$, in contrast to the Laplace operator. Still, the specific choice of the scaling of the nonlocal operators $\Gamma_{\chi_j}$ allows us to recover compactness, which is obtained by adapting the compactness result derived in~\cite[Theorem~4]{BBM2001} and further developed in~\cite[Theorem~6.11]{AVMRTM2010} and~\cite[Theorem~1.2]{Pon2004} to a time-dependent setting, thereby obtaining a nonlocal version of the compactness results in~\cite{Sim1987}. Similar results are already obtained in~\cite{CES2023, ElSk2023} when $\Omega$ is the $n$-dimensional torus and the version we provide in Proposition~\ref{propA.1} below is adapted to a bounded domain and nonlocal operators of the form~\eqref{nlo}.  As a final comment, let us point out that Theorem~\ref{thm1.3} is likely to be true in one space dimension $n=1$ as well. However, the analysis performed in~\cite{Pon2004}, on which the compactness argument used in the proof of Theorem~\ref{thm1.3} relies, is slightly different in that case and we have chosen not to include the one-dimensional setting for simplicity.

\section{Well-posedness: Proof of Theorem~\ref{thm1.1}}\label{sec2}

We begin this section with the properties of the linear term in~\eqref{CNLGS}. Specifically, given a non-negative measurable function $\gamma: \Omega\times\Omega \to [0,\infty)$ and a measurable function $z:\Omega\to\mathbb{R}$, we define
\begin{equation}
	\Gamma_{\gamma} z(x) := \int_\Omega \gamma(x,y) \big( z(y)-z(x) \big)\,\mathrm{d}y\,, \qquad x\in\Omega\,, \label{b1}
\end{equation}
whenever it makes sense.

\begin{lemma}\label{lem2.1}
	Assume that there is $\gamma_\infty>0$ such that
	\begin{equation}
		\int_{\Omega} \gamma(y,x)\,\mathrm{d} y = \int_{\Omega} \gamma(x,y)\,\mathrm{d} y \le  \gamma_\infty<\infty\,, \qquad x\in\Omega\,. \label{b2}
	\end{equation}
	Then $\Gamma_{\gamma}\in\mathcal{L}(X)$ with $\|\Gamma_{\gamma}\|_{\mathcal{L}(X)} \le 2 \gamma_\infty$, and $\Gamma_{\gamma}$ generates a uniformly continuous semigroup $\big( e^{t\Gamma_{\gamma}} \big)_{t\ge 0}$ on $X$ satisfying
	\begin{equation}
		\|e^{t\Gamma_{\gamma}}\|_{\mathcal{L}(X)}\le 1\,,\qquad t\ge0\,. \label{b3}
	\end{equation}
	Moreover, for $t\ge 0$,
	\begin{equation}
		e^{t\Gamma_{\gamma}}\ge 0 \;\;\text{ and }\;\; e^{t\Gamma_{\gamma}} 1= 1\,. \label{b4}
	\end{equation} 
	If, in addition, $\gamma\in C(\bar{\Omega}\times\bar{\Omega})$, then the above statements remain true with $C(\bar{\Omega})$ replacing $X$.
\end{lemma}

\begin{proof}
	It readily follows from~\eqref{b2} that, for $z\in X$ and $x\in\Omega$, 
	\begin{equation*}
		|\Gamma_{\gamma} z(x)| \le \int_\Omega \gamma(x,y) |z(y)-z(x)|\,\mathrm{d} y \le 2 \gamma_\infty \|z\|_\infty\,.
	\end{equation*}
	Hence $\Gamma_{\gamma}\in\mathcal{L}(X)$ with $\|\Gamma_{\gamma}\|_{\mathcal{L}(X)} \le 2 \gamma_\infty$ and $\Gamma_{\gamma}$ generates a uniformly continuous semigroup $\big( e^{t\Gamma_{\gamma}} \big)_{t\ge 0}$ on $X$. In order to prove that $\big( e^{t\Gamma_{\gamma}} \big)_{t\ge 0}$ is a semigroup of contractions on $X$, we pick $z^0\in X$ and set $z(t) := e^{t\Gamma_{\gamma}} z^0$ for $t\ge 0$. We then note that, for $M\in\mathbb{R}$ and $t>0$,
	\begin{align*}
		\frac{\mathrm{d}}{\mathrm{d} t} \int_\Omega \big(z(t,x)-M\big)_+\, \mathrm{d} x
		&=\int_\Omega (\Gamma_{\gamma} z(t))(x)\, \mathrm{sign}_+\big(z(t,x)-M\big) \,\mathrm{d} x\\
		&=\int_\Omega  \int_{\Omega} \gamma(x,y) \big(z(t,y)-z(t,x)\big)\, \mathrm{sign}_+ \big(z(t,x)-M\big)\,\mathrm{d} y\mathrm{d} x\\
		&=\int_\Omega  \int_{\Omega} \gamma(x,y) \big(z(t,y)-M\big) \,\mathrm{sign}_+\big(z(t,x)-M\big) \,\mathrm{d} y\mathrm{d} x\\
		&\quad - \int_\Omega  \int_{\Omega} \gamma(x,y) \big(z(t,x)-M\big)_+\,\mathrm{d} y\mathrm{d} x\\
		&\le  \int_\Omega  \big(z(t,y)-M \big)_+ \int_{\Omega} \gamma(x,y)\,\mathrm{d} x\mathrm{d} y\\
		&\quad -\int_\Omega  \big(z(t,x)-M \big)_+ \int_{\Omega} \gamma(x,y)\,\mathrm{d} y\mathrm{d} x\\
		&=  \int_\Omega  \big(z(t,x)-M \big)_+ \int_{\Omega} \gamma(y,x)\,\mathrm{d} y\mathrm{d} x\\
		&\quad -\int_\Omega  \big(z(t,x)-M\big)_+ \int_{\Omega} \gamma(x,y)\,\mathrm{d} y\mathrm{d} x\\
		& \le 0\,,
	\end{align*}
	where we used~\eqref{b2} to obtain the last inequality. Consequently, 
	\begin{equation*}
		\frac{\mathrm{d}}{\mathrm{d}t} \int_\Omega \big(z(t,x)-M\big)_+\, \mathrm{d} x \le 0\,, \qquad t>0\,,
	\end{equation*}
	and thus
	\begin{equation}
		\big( e^{t\Gamma_{\gamma}} z^0 \big)(x) \le M + (z^0(x) - M)_+\,, \qquad (t,x)\in [0,\infty)\times\Omega\,. \label{b5}
	\end{equation}
	Applying~\eqref{b5} with $M=\|z^0\|_\infty$ to $z^0$ and $(-z^0)$, we conclude that $\|e^{t\Gamma_{\gamma}}z^0\|_\infty\le \|z^0\|_\infty$ for $t\ge 0$. 
	
	Consider now $z^0\in X^+$ and apply~\eqref{b5} with $M=0$ to $(-z^0)$ to obtain $-\big( e^{t\Gamma_{\gamma}} z^0\big)(x)\le 0$ for $(t,x)\in [0,\infty)\times\Omega$, which shows the positivity of the semigroup claimed in~\eqref{b4}. Finally, the second property stated in~\eqref{b4} is an obvious consequence of the definition~\eqref{b1} of~$\Gamma_{\gamma}$.
\end{proof}

We now show the well-posedness of~\eqref{CNLGS} in $X^+\times X^+$.

\begin{proposition}\label{prop2.2}
	Assume that $\gamma_\ell: \Omega\times\Omega\to [0,\infty)$ is a measurable function satisfying~\eqref{gammaN} for $\ell\in \{1,2\}$. For every $(u^0,v^0)\in X^+\times X^+$ there is a unique non-negative global solution 
	\begin{equation*}
		(u,v)\in C^1\big([0,\infty),X^+\times X^+\big)
	\end{equation*}
	to~\eqref{CNLGS}. In fact, the mapping $(u^0,v^0)\mapsto (u,v)$ defines a global semiflow on $X^+\times X^+$. If $\gamma_\ell\in C\big(\bar{\Omega}\times\bar{\Omega}\big)$ for $\ell\in \{1,2\}$, then the same result is true when replacing $X$ by~$C\big(\bar{\Omega}\big)$.
\end{proposition}

\begin{proof}
	We set $z=(u,v)$, $\Gamma_\ell:=\Gamma_{\gamma_\ell}$ for $\ell\in\{1,2\}$, 
	\begin{equation*}
		A:=\left(\begin{matrix}
			d_1\Gamma_{1} & 0\\ 0& d_2\Gamma_{2}
			\end{matrix}\right) \in \mathcal{L}(X^2) \;\;\text{ and }\;\; 
		F(z):=\left(\begin{matrix}
				F_1(z) \\ F_2(z)
			\end{matrix}\right):=\left(\begin{matrix}
				-uv^2+f(1-u) \\ uv^2-(f+\kappa)v
			\end{matrix}\right)\,.
	\end{equation*}
	Given $z^0=(u^0,v^0)\in X^2$, the initial value problem~\eqref{CNLGS} is equivalent to
	\begin{equation}
		\frac{\mathrm{d}z}{\mathrm{d}t}=A z +F(z)\,, \quad t\ge 0\,,\qquad z(0)=z^0\,. \label{b6}
	\end{equation}
	Since $A\in \mathcal{L}(X^2)$ and $F\in C_b^{1-}(X^2,X^2)$, the initial value problem~\eqref{b6} has a unique solution
	\begin{equation*}
		z= z(\cdot;z^0)=(u,v)\in C^1([0,T_m),X^2)
\end{equation*} 
defined on a maximal time interval $[0,T_m)$ with $T_m\in (0,\infty]$ and $(t,z^0)\mapsto z(t;z^0)$ defines a semiflow on $X^2$. Moreover, if $T_m<\infty$, then
	\begin{equation}\label{global}
		\lim_{t\to T_m}\|z(t)\|_{X^2}=\infty\,.
	\end{equation}
	We next observe that, for every $R>0$, there is $\xi(R) := f + \max\{\kappa,R^2\}>0$ such that
	\begin{equation}
		F(z)+\xi(R)z\ge 0 \;\;\text{ for }\;\; z\ge 0 \;\;\text{ with }\;\; \|z\|_{X^2}\le R\,. \label{locpos}
	\end{equation}
	By~\eqref{b6}, we have the representation formula
	\begin{equation*}
		z(t) = e^{t(A-\xi(R))} z^0 + \int_0^t e^{(t-s)(A-\xi(R))} \big[ F(z(s)) + \xi(R) z(s) \big]\,\mathrm{d}s\,,
	\end{equation*}
	and hence deduce from~\eqref{b4} and~\eqref{locpos} that
	\begin{equation}
		z(t)=(u(t),v(t)) \in X^+\times X^+ \;\;\text{ for }\;\; t\in [0,T_m) \;\;\text{ whenever }\;\; z^0=(u^0,v^0) \in X^+\times X^+\,. \label{p1}
	\end{equation}
	
	In order to prove global existence for $(u^0,v^0)\in X^+\times X^+$, we note that $w:=u-1$ satisfies
	\begin{equation}\label{uw}
		\begin{split}
			\partial_t w & = d_1\Gamma_1 w - (1+w)v^2-fw \,,\qquad (t,x)\in [0,T_m)\times\Omega\,, \\  
			w(0) & =u^0-1 \,, \qquad x\in\Omega\,.
		\end{split}
	\end{equation}
	Thus, for $t\in [0,T_m)$, it follows from~\eqref{p1} that
	\begin{equation*}
		w(t) = e^{t(d_1\Gamma_1-f)} w(0) - \int_0^t e^{(t-s)(d_1\Gamma_1-f)} \big((1+w(s)) v(s)^2\big) \,\mathrm{d} s \le e^{t(d_1\Gamma_1-f)} w(0)\,.
	\end{equation*}
	Since $w(0)\le (\|u^0\|_\infty-1)_+$ we infer from~\eqref{b5} that
	\begin{equation*}
		w(t)\le e^{t(d_1\Gamma_1-f)} w(0)\le e^{-ft}\big(\|u^0\|_\infty-1\big)_+\,,\qquad t\in [0,T_m)\,.
	\end{equation*}
	Therefore,
	\begin{align}\label{APuu}
		0\le u(t,x)\le 1+e^{-ft}\big(\|u^0\|_\infty-1\big)_+\,, \qquad (t,x) \in [0,T_m)\times\Omega\,,
	\end{align}
	and in particular, 
	\begin{equation}\label{b8}
		\|u(t)\|_\infty\le\|u^0\|_\infty + 1\,,\quad t\in [0,T_m)\,.
	\end{equation} 
	
	Next, $h:=u+v$ solves
	\begin{align*}
		\partial_t h &=d_2\Gamma_2 h +(d_1 \Gamma_1 - d_2\Gamma_2) u -f h + f-\kappa v \;\;\text{ in }\;\; (0,T_m)\times\Omega\,, \\
		h(0) & = u^0 + v^0 \;\;\text{ in }\;\; \Omega\,,
	\end{align*}
	hence
	\begin{equation}
		h(t) = e^{t(d_2\Gamma_2-f)} h(0) + \int_0^t e^{(t-s)(d_2\Gamma_2-f)} \big[ (d_1\Gamma_1 - d_2 \Gamma_2) u(s) + f - \kappa v(s) \big]\,\mathrm{d}s \label{p2}
	\end{equation}
	for $t\in [0,T_m)$. According to~\eqref{gammaN}, \eqref{b8}, and Lemma~\ref{lem2.1}, we have
	\begin{align*}
		\|(d_1\Gamma_1 - d_2\Gamma_2)u(s)\|_\infty & \le |d_1-d_2| \|\Gamma_1 u(s)\|_\infty + d_2 \|(\Gamma_1-\Gamma_2)u(s)\|_\infty \\
		& \le 2 \gamma_\infty |d_1-d_2| \|u(s)\|_\infty \\
		& \qquad + d_2 \mathop{\mathrm{ess\, sup}}_{x\in\Omega} \left| \int_\Omega (\gamma_1-\gamma_2)(x,y) (u(s,y)-u(s,x))\,\mathrm{d}y \right| \\
		& \le 2 \gamma_\infty |d_1-d_2| \|u(s)\|_\infty + 2 m d_2 \|u(s)\|_\infty \\
		& \le 2 \big( |d_1-d_2| \gamma_\infty + m d_2 \big) (1+\|u^0\|_\infty)\,,
	\end{align*}
	where
	\begin{equation*}
		m := \mathop{\mathrm{ess\, sup}}_{x\in\Omega} \int_\Omega |(\gamma_1-\gamma_2)(x,y)| \,\mathrm{d}y \le 2 \gamma_\infty\,.
	\end{equation*}
	Combining~\eqref{p2}, the above estimate, the non-negativity of $u$, and Lemma~\ref{lem2.1} gives, for $t\in [0,T_m)$,
	\begin{align}
		0 \le h(t) & \le e^{-ft} \|u^0+v^0\|_\infty + \int_0^t e^{-f(t-s)} \left[ 2 \big( |d_1-d_2| \gamma_\infty + m d_2 \big) (1+\|u^0\|_\infty) + f \right]\,\mathrm{d}s \nonumber \\
		& \le e^{-tf} \|u^0+v^0\|_\infty + \frac{1-e^{-tf}}{f} \left[ 2 \big( |d_1-d_2|\,\gamma_\infty + m d_2 \big) (1+\|u^0\|_\infty)+f \right]\,. \label{b9}
	\end{align}
Owing to~\eqref{b8} and~\eqref{b9}, we conclude $T_m=\infty$ according to~\eqref{global}.
\end{proof}

The proof of Proposition~\ref{prop2.2} actually provides the boundedness of solutions to~\eqref{CNLGS}.

\begin{corollary}\label{cor2.3}
	Assume that $\gamma_\ell: \Omega\times\Omega\to [0,\infty)$ is a measurable function satisfying~\eqref{gammaN} for $\ell\in \{1,2\}$ and consider $(u^0,v^0)\in X^+\times X^+$. The corresponding solution $(u,v)$ to~\eqref{CNLGS} satisfies
	\begin{equation}\label{ubu}
		0\le u(t,x)\le 1+e^{-ft}\big(\|u^0\|_\infty-1\big)_+
	\end{equation}
	and
	\begin{equation}
		0\le (u+v)(t,x) \le e^{-tf} \|u^0+v^0\|_\infty + \frac{1-e^{-tf}}{f} \left[ 2 \big( |d_1-d_2|\,\gamma_\infty + m d_2 \big)(1+\|u^0\|_\infty) + f \right] \label{ubv}
	\end{equation}
	for $(t,x)\in [0,\infty)\times\Omega$, where
	\begin{equation*}
		m=\mathop{\mathrm{ess\ sup}}_{x\in\Omega} \int_\Omega |(\gamma_1 - \gamma_2)(x,y)| \,\mathrm{d} y \le 2\gamma_\infty\,.
	\end{equation*}
	In particular,
	\begin{equation*}
		\limsup_{t\to\infty}\|u(t)\|_\infty\le 1
	\end{equation*}
	and, if $\|u^0\|_\infty\le 1$, then $\|u(t)\|_\infty\le 1$ for $t\ge 0$.
\end{corollary}

\begin{proof}
	The bounds on $u$ and $u+v$ are established in~\eqref{b8} and~\eqref{b9}, respectively, while the last statement of Corollary~\ref{cor2.3} readily follows from~\eqref{ubu}.
\end{proof}

\begin{proof}[Proof of Theorem~\ref{thm1.1}]
	Gathering the outcome of Proposition~\ref{prop2.2} and Corollary~\ref{cor2.3} gives Theorem~\ref{thm1.1}.
\end{proof}

\section{Convergence to spatially homogeneous steady states: Proof of Theorem~\ref{thm1.2}}\label{sec3}

Let us begin with the local asymptotic stability of $(1,0)$. 

\begin{proposition}\label{prop3.1}
		Assume that $\gamma_\ell: \Omega\times\Omega\to [0,\infty)$ is a measurable function satisfying~\eqref{gammaN} for $\ell\in \{1,2\}$. Then the semi-trivial steady state $(1,0)$ is locally asymptotically stable in $X^2$.
\end{proposition}

\begin{proof}
	Set $w=u-1$. Then \eqref{CNLGS} is equivalent to
	\begin{subequations}\label{GSw}
		\begin{align}
			\partial_t w&=d_1\Gamma_{\gamma_1} w-(1+w)v^2-fw \,,  &&w(0)=u^0-1\,, \\
			\partial_t v&=d_2\Gamma_{\gamma_2} v+(1+w)v^2-(f+\kappa)v \,, && v(0)=v^0 \,, 
		\end{align}
	\end{subequations}
	so that, with $\zeta=(w,v)$,
	\begin{equation*}
		\zeta(t)=e^{tA_0}\zeta^0+\int_0^t e^{(t-s)A_0} R(\zeta(s))\,\mathrm{d} s\,,\quad t\ge 0\,,
	\end{equation*}
	where
	\begin{equation*}
		A_0:=\left(\begin{matrix}
		d_1 \Gamma_{\gamma_1} -f & 0\\ 0& d_2 \Gamma_{\gamma_2}-(f+\kappa)
		\end{matrix}\right) \in \mathcal{L}(X^2)\,,\qquad \|e^{tA_0}\|_{\ml(X^2)}\le e^{-ft}\,,\quad t\ge 0\,,
	\end{equation*}
	according to Lemma~\ref{lem2.1} and $R(\zeta)=o(\|\zeta\|_{X^2})$ as $\zeta\to 0$. The principle of linearized stability then ensures that $(w,v)=(0,0)$ is locally asymptotically stable in $X^2$ for~\eqref{GSw}, see \cite[Theorem~15.3]{Aman1990} for instance.
\end{proof}

\begin{remark}\label{rem3.2}
	If $\gamma_\ell\in C(\bar{\Omega}\times\bar{\Omega})$ for $\ell\in \{1,2\}$, then the semi-trivial steady state $(1,0)$ is also locally asymptotically stable in $C(\bar{\Omega},\mathbb{R}^2)$.
\end{remark}

We next finish the proof of Theorem~\ref{thm1.2} by providing a more quantitative stability result for the semi-trivial steady state~$(1,0)$ and actually identify an invariant set for the flow associated with~\eqref{CNLGS} in $X^+\times X^+$.

\begin{proof}[Proof of Theorem~\ref{thm1.2}]
We have already shown in Proposition~\ref{prop3.1} the local stability of the semi-trivial steady state~$(1,0)$. We next refine this local stability result when the initial conditions satisfy~\eqref{init}. To this end, we first recall from~Corollary~\ref{cor2.3}  that $\|u(t)\|_\infty\le 1+\delta$ for $t\ge 0$. Owing to~\eqref{init}, we may choose $\varepsilon>0$ such that 
	\begin{equation*}
		\|v^0\|_\infty<\frac{f+\kappa-\varepsilon}{1+\delta}
	\end{equation*}
	and the time continuity of~$v$ implies that
	\begin{equation*}
		t_0 := \inf\left\{  t\ge 0\ :\ \|v(t)\|_\infty\ge \frac{f+\kappa-\varepsilon}{1+\delta} \right\}>0\,.
	\end{equation*}
	Then
	\begin{align}\label{APu2w}
		F_2(u,v)(s)=u(s)v(s)^2-(f+\kappa)v(s)\le (u(s)v(s)-f-\kappa)v(s)\le -\varepsilon v(s)\,,\quad s\in [0,t_0)\,,
	\end{align}
	so that, thanks to~\eqref{vNLGS} and Lemma~\ref{lem2.1},
	\begin{equation*}
		v(t)=e^{t d_2\Gamma_{\gamma_2}}v^0+\int_0^te^{(t-s) d_2\Gamma_{\gamma_2}}F_2(u(s),v(s))\,\mathrm{d} s\le e^{t d_2\Gamma_{\gamma_2}}v^0\,,\quad t\in [0,t_0)\,.
	\end{equation*}
	We now infer from Lemma~\ref{lem2.1} and the above inequality that
	\begin{equation*}
		\|v(t)\|_\infty\le \|v^0\|_\infty<\frac{f+\kappa-\varepsilon}{1+\delta}\,,\quad t\in [0,t_0)\,,
	\end{equation*}
	and therefore $t_0=\infty$; that is,
	\begin{equation*}
		\|v(t)\|_\infty\le \|v^0\|_\infty<\frac{f+\kappa-\varepsilon}{1+\delta}\,,\quad t\ge 0\,.
	\end{equation*}
	Using again~\eqref{vNLGS} and~\eqref{APu2w}, we deduce  
	\begin{equation*}
		\partial_t v=d_2\Gamma_{\gamma_2} v+uv^2-(f+\kappa)v\le d_2\Gamma_{\gamma_2} v-\varepsilon v\,,\qquad t\ge 0\,,
	\end{equation*}
	and  Lemma~\ref{lem2.1} ensures that
	\begin{equation}
		\|v(t)\|_\infty\le e^{-\varepsilon t}\,\|v^0\|_\infty\,, \qquad t\ge 0\,. \label{p7}
	\end{equation}
	In particular,
	\begin{equation*}
		\lim_{t\to\infty} \|v(t)\|_\infty = 0\,.
	\end{equation*}
	Furthermore, it follows from~\eqref{ubu} and~\eqref{p7} that, for every $\eta>0$, there is $t_\eta>0$ such that 
	\begin{equation*}
		0\le u(s,x))v(s,x)^2\le \eta\,,\qquad (s,x)\in (t_\eta,\infty)\times\Omega\,.
	\end{equation*}
	Consequently, invoking~\eqref{uNLGS} and Lemma~\ref{lem2.1},
	\begin{align*}
		u(t+t_\eta) - 1 &=e^{t(d_1\Gamma_{\gamma_1}-f)} \big[ u(t_\eta) - 1 \big] - \int_{t_\eta}^{t+t_\eta} e^{(t+t_\eta-s)(d_1\Gamma_{\gamma_1}-f)} \big(u(s)v(s)^2\big)\mathrm{d} s\\ 
		&\ge - \left\| e^{t(d_1\Gamma_{\gamma_1}-f)} \big[ u(t_\eta) - 1 \big] \right\|_\infty -\int_{t_\eta}^{t+t_\eta} \|e^{(t+t_\eta-s)(d_1\Gamma_{\gamma_1}-f)}\eta\|_\infty\, \mathrm{d} s\\
		&\ge -e^{-ft} \left\| u(t_\eta) - 1  \right\|_\infty -\frac{1-e^{-ft}}{f}\eta
	\end{align*}
	for $t\ge 0$. We may then let $t\to\infty$ and conclude that 
	\begin{equation*}
		\liminf_{t\to\infty} \mathop{\mathrm{ess\ inf}}_{x\in\Omega} \{u(t,x)-1\} = \liminf_{t\to\infty} \mathop{\mathrm{ess\ inf}}_{x\in\Omega} \{u(t+t_\eta,x)-1\}  \ge - \frac{\eta}{f}\,.
	\end{equation*}
	Since $\eta>0$ was arbitrary and since
	\begin{equation*}
		\mathop{\mathrm{ess\ sup}}_{x\in\Omega} \{u(t,x)-1\} \le  e^{-ft}\big(\|u^0\|_\infty-1\big)_+\,,\quad t\ge 0\,,
	\end{equation*}
	by Corollary~\ref{cor2.3}, we end up with
	\begin{equation*}
		\lim_{t\to\infty} \|u(t)-1\|_\infty= 0\,,
	\end{equation*}
	which completes the proof.
\end{proof}

\begin{remark}\label{rem3.3}
	Assume that $\gamma_\ell: \Omega\times\Omega\to [0,\infty)$ is a measurable function satisfying~\eqref{gammaN} for $\ell\in \{1,2\}$ and consider $(u^0,v^0)\in X^+\times X^+$ such that 
	\begin{equation*} 
		\|u^0\|_\infty\le 1+\delta \;\;\text{ and }\;\; \|v^0\|_\infty\le \frac{f+\kappa}{1+\delta}
	\end{equation*} 
	for some $\delta\ge 0$. Then the solution~$(u,v)$ to~\eqref{CNLGS} satisfies
	\begin{equation*}
		\|u(t)\|_\infty\le 1+\delta\,,\qquad \|v(t)\|_\infty\le \|v^0\|_\infty\,, \qquad t\ge 0\,.
	\end{equation*}
	Indeed, this property follows from applying Theorem~\ref{thm1.2} to the initial value $(u^0,\theta v^0)$ with $\theta\in (0,1)$ and then using the continuous dependence of the solution on the initial value to let~$\theta\to 1$.
\end{remark}

\section{Diffusive limit: Proof of Theorem~\ref{thm1.3}}\label{sec4}

This section is devoted to a mathematical proof of the connection between the Gray-Scott model~\eqref{GS} and its nonlocal version~\eqref{CNLGS} with the suitable choice of nonlocal operators described in Theorem~\ref{thm1.3}. We assume that $n\ge 2$ and consider a non-negative radially symmetric function  $\varphi\in C_0^\infty(\mathbb{R}^n)$ with compact support in $B_1(0)$ and a non-increasing radial profile. For each integer $j\ge 1$, we define $\chi_{j}: \Omega\times\Omega\to [0,\infty)$ by
\begin{equation}
	\chi_{j}(x,y) := j^{n+2} \varphi(j(x-y))\,, \qquad (x,y)\in \Omega\times\Omega\,, \label{d1000}
\end{equation}
and, for $z\in X = L_\infty(\Omega)$, 
\begin{equation}
	\Gamma_{\chi_j} z(x) := \int_\Omega \chi_{j}(x,y) \big(z(y)-z(x)\big)\,\mathrm{d}y\,, \qquad x\in\Omega\,. \label{d100}
\end{equation}
We also set
\begin{equation}
	Y_j[z] := \int_{\Omega\times\Omega} \chi_{j}(x,y) [z(x)-z(y)]^2\,\mathrm{d}(x,y)\,, \qquad z\in L_2(\Omega)\,. \label{OpAux}
\end{equation}

We fix $(u^0,v^0)\in X^+\times X^+$ and denote the corresponding solution to~\eqref{CjNLGS} by
\begin{equation*}
	(u_j,v_j)\in C^1([0,\infty),X^+\times X^+)\,.
\end{equation*}
The first step of the proof is the derivation of several estimates which do not depend on $j\ge 1$. Throughout this section, $C$ and $C_i$, $i\ge 0$, denote positive constants which only depend on $n$, $\Omega$, $d_1$, $d_2$, $\varphi$, $f$, $\kappa$, and $(u^0,v^0)$. Dependence upon additional parameters will be indicated explicitly.

\subsection{Estimates}\label{sec4.1}

Recall that, for $j\ge 1$, 
\begin{equation}
	0 \le u_j(t,x) \le 1 + e^{-ft} \big( \|u^0\|_\infty -1 \big)_+ \le C_0 := 1 + \|u^0\|_\infty\,, \qquad t\ge 0\,,  \label{d3}
\end{equation}
by Corollary~\ref{cor2.3}. We next proceed as in \cite[Lemma~2]{You2008} to derive $L_2$-estimates on $(u_j,v_j)$.

\begin{lemma}\label{lem4.1}
	For $t>0$ and $j\ge 1$, 
	\begin{subequations}\label{d60}
	\begin{equation}
		\|u_j(t)\|_2^2 + \|v_j(t)\|_2^2 \le C_1\label{d6a}
	\end{equation}
	and
	\begin{equation}
		\int_0^t \left( Y_j[u_j(s)] + Y_j[v_j(s)] + \|(u_j v_j)(s)\|_2^2 \right)\,\mathrm{d}s \le C_1 (1+t)\,. \label{d6b} 
	\end{equation}
\end{subequations}
\end{lemma}

\begin{proof}
	Let $t\ge 0$ and $j\ge 1$. We first infer from~\eqref{ujNLGS}, \eqref{d100}, and the non-negativity of $u_j$ that
	\begin{align*}
		\frac{1}{2} \frac{\mathrm{d}}{\mathrm{d}t} \|u_j(t)\|_2^2 & + d_1 \int_{\Omega\times\Omega} \chi_{j}(x,y) u_j(t,x) [u_j(t,x)-u_j(t,y)]\,\mathrm{d}(x,y) \\
		& + \|(u_j v_j)(t)\|_2^2 + f \|u_j(t)\|_2^2 = f \|u_j(t)\|_1\,.
	\end{align*}
	Since
	\begin{align*}
		& \int_{\Omega\times\Omega} \chi_{j}(x,y) u_j(t,x) [u_j(t,x)-u_j(t,y)]\,\mathrm{d}(x,y) \\
		& \qquad = \frac{1}{2} \int_{\Omega\times\Omega} \chi_{j}(x,y) u_j(t,x) [u_j(t,x)-u_j(t,y)]\,\mathrm{d}(x,y) \\
		& \qquad\qquad + \frac{1}{2} \int_{\Omega\times\Omega} \chi_{j}(y,x) u_j(t,y) [u_j(t,y)-u_j(t,x)]\,\mathrm{d}(x,y) \\
		& \qquad = \frac{Y_j[u_j(t)]}{2} 
	\end{align*}
	by the symmetry of $\chi_{j}$ and 
	\begin{equation*}
		f \|u_j(t)\|_1 \le \frac{f}{2} \|u_j(t)\|_2^2 + \frac{f}{2} |\Omega|
	\end{equation*}
	by Young's inequality, we conclude that
	\begin{equation}
		\frac{\mathrm{d}}{\mathrm{d}t} \|u_j(t)\|_2^2 + d_1 Y_j[u_j(t)] + 2 \|(u_j v_j)(t)\|_2^2 + f \|u_j(t)\|_2^2 \le f |\Omega|\,. \label{d4}
	\end{equation}
	We next recall that $h_j := u_j+v_j$ solves
	\begin{equation}
		\partial_t h_j = d_2 \Gamma_{\chi_j}h_j + (d_1-d_2) \Gamma_{\chi_j}u_j + f(1-h_j) - \kappa v_j \;\;\text{ in }\;\; (0,\infty)\times\Omega\,. \label{d101}
	\end{equation}
	Then, thanks to the non-negativity of $\kappa v_j$, 
	\begin{align*}
		& \frac{1}{2} \frac{\mathrm{d}}{\mathrm{d}t} \|h_j(t)\|_2^2 + d_2 \int_{\Omega\times\Omega} \chi_{j}(x,y) h_j(t,x) [h_j(t,x)-h_j(t,y)]\,\mathrm{d}(x,y) + f \|h_j(t)\|_2^2 \\
		& \qquad \le (d_1-d_2) \int_{\Omega\times\Omega} \chi_{j}(x,y) h_j(t,x) [u_j(t,y)-u_j(t,x)]\,\mathrm{d}(x,y) + f \|h_j(t)\|_1\,.
	\end{align*}
	We again use the symmetry of $\chi_{j}$, along with Young's inequality, to deduce that
	\begin{equation*}
		d_2 \int_{\Omega\times\Omega} \chi_{j}(x,y) h_j(t,x) [h_j(t,x)-h_j(t,y)]\,\mathrm{d}(x,y) = \frac{d_2}{2} Y_j[h_j(t)]
	\end{equation*}
	and
	\begin{align*}
		& (d_1-d_2) \int_{\Omega\times\Omega} \chi_{j}(x,y) h_j(t,x) [u_j(t,y)-u_j(t,x)]\,\mathrm{d}(x,y) \\
		& \qquad = \frac{d_2-d_1}{2} \int_{\Omega\times\Omega} \chi_{j}(x,y) [h_j(t,x) - h_j(t,y)] [u_j(t,x)-u_j(t,y)]\,\mathrm{d}(x,y) \\
		& \qquad \le \frac{d_2}{4} Y_j[h_j(t)]+ \frac{(d_1-d_2)^2}{4d_2} Y_j[u_j(t)]\,.
	\end{align*}
	Also,
	\begin{equation*}
		f \|h_j(t)\|_1 \le \frac{f}{2} \|h_j(t)\|_2^2 + \frac{f}{2} |\Omega|\,.
	\end{equation*}
	Gathering the above estimates leads us to 
	\begin{equation}
		\frac{\mathrm{d}}{\mathrm{d}t} \|h_j(t)\|_2^2 + \frac{d_2}{2} Y_j[h_j(t)] + f \|h_j(t)\|_2^2 \le \frac{(d_1-d_2)^2}{2d_2} Y_j[u_j(t)] + f |\Omega|\,. \label{d5}
	\end{equation}
	Setting $C_2 := (d_1 d_2)/[1 +(d_1-d_2)^2]$, we infer from~\eqref{d4} and~\eqref{d5} that
	\begin{align*}
		& \frac{\mathrm{d}}{\mathrm{d}t} \left[ \|u_j(t)\|_2^2 + C_2\|h_j(t)\|_2^2 \right] + \left( d_1 - \frac{C_2 (d_1-d_2)^2}{2 d_2} \right) Y_j[u_j(t)] + \frac{d_2 C_2}{2} Y_j[h_j(t)] \\
		& \qquad + 2 \|(u_j v_j)(t)\|_2^2 + f \left( \|u_j(t)\|_2^2 + C_2 \|h_j(t)\|_2^2 \right) \le f (1+C_2) |\Omega|\,.
	\end{align*}
	Hence,
	\begin{equation}
		\begin{split}
			& \frac{\mathrm{d}}{\mathrm{d}t} \left[ \|u_j(t)\|_2^2 + C_2\|h_j(t)\|_2^2 \right] + f \left(  \|u_j(t)\|_2^2 + C_2 \|h_j(t)\|_2^2 \right) \\
			& \qquad\qquad + \frac{d_1}{2} Y_j[u_j(t)]+ \frac{d_2 C_2}{2} Y_j[h_j(t)] + 2 \|(u_j v_j)(t)\|_2^2 \le f (1+C_2) |\Omega|\,.
		\end{split}\label{d6}
	\end{equation}
	A first consequence of~\eqref{d6} is that
	\begin{equation*}
		\|u_j(t)\|_2^2 + C_2\|h_j(t)\|_2^2 \le e^{-ft} \left[ \|u^0\|_2^2 + C_2\|u^0+v^0\|_2^2 \right] + (1+C_2) |\Omega| \big( 1 - e^{-ft} \big)\,,
	\end{equation*}
	from which~\eqref{d6a} readily follows, since $v_j=h_j-u_j$. We next integrate~\eqref{d6} with respect to time to obtain~\eqref{d6b}, using again the relation $v_j = h_j-u_j$, and complete the proof.
\end{proof}

We now turn to estimates on $\partial_t u_j$ and $\partial_t v_j$.

\begin{lemma}\label{lem4.2}
	For $t>0$ and $j\ge 1$, 
	\begin{equation}
		\int_0^t \left[ \|\partial_t u_j(s)\|_{(W_{n+1}^1)'}^2 + \|\partial_t v_j(s)\|_{(W_{n+1}^1)'}^2 \right]\,\mathrm{d}s \le C_3(1+t)\,. \label{d7}
	\end{equation}
\end{lemma}

\begin{proof}
Since $W_{n+1}^1(\Omega)$ is continuously embedded in~$L_\infty(\Omega)$, there is $C_4>0$ such that
\begin{equation}
	\|\vartheta\|_{\infty} \le C_4 \|\vartheta\|_{W_{n+1}^1}\,, \qquad \vartheta\in W_{n+1}^1(\Omega)\,. \label{d102}
\end{equation} 
Let $\vartheta\in W_{n+1}^1(\Omega)$, $t>0$, and $j\ge 1$. We infer from~\eqref{ujNLGS}, \eqref{d3}, \eqref{d6a}, \eqref{d102}, and the symmetry of~$\chi_{j}$ that
\begin{align*}
	\left| \int_\Omega \vartheta(x) \partial_t u_j(t,x)\,\mathrm{d}x \right| & \le d_1 \left| \int_{\Omega\times\Omega} \chi_{j}(x,y) \vartheta(x) [u_j(t,y)]-u_j(t,x)]\,\mathrm{d}(x,y) \right| \\
	& \qquad + \|\vartheta\|_\infty \|u_j(t)\|_\infty \|v_j(t)\|_2^2 + f \|\vartheta\|_\infty |\Omega| \left( 1 + \|u_j(t)\|_\infty \right) \\
	& \le \frac{d_1}{2} \left| \int_{\Omega\times\Omega} \chi_{j}(x,y) [\vartheta(x) - \vartheta(y)] [u_j(t,x)]-u_j(t,y)]\,\mathrm{d}(x,y) \right| + C \|\vartheta\|_{W_{n+1}^1}\,.
\end{align*}
	Since the support of $\chi_{j}$ is a subset of $\{(x,y)\in\mathbb{R} ^{2n}: |x-y|\le C/j\}$, it follows from H\"older's inequality and \cite[Theorem~1]{BBM2001} that
	\begin{align*}
		& \left| \int_{\Omega\times\Omega} \chi_{j}(x,y) [\vartheta(x) - \vartheta(y)] [u_j(t,x)]-u_j(t,y)]\,\mathrm{d}(x,y) \right| \\ 
		& \qquad \le Y_j[u_j(t)]^{1/2} \left( \int_{\Omega\times\Omega} \chi_{j}(x,y) [\vartheta(x) - \vartheta(y)]^2\,\mathrm{d}(x,y) \right)^{1/2} \\
		& \qquad \le \frac{C}{j} Y_j[u_j(t)]^{1/2} \left( \int_{\Omega\times\Omega} \frac{\chi_{j}(x,y)}{|x-y|^2} [\vartheta(x) - \vartheta(y)]^2\,\mathrm{d}(x,y) \right)^{1/2} \\
		& \qquad \le C \|\varphi\|_{L_1(\mathbb{R}^n)}^{1/2} \|\vartheta\|_{W_2^1} Y_j[u_j(t)]^{1/2}  \le C \|\vartheta\|_{W_{n+1}^1}Y_j[u_j(t)]^{1/2} \,.
	\end{align*}
	Consequently, 
	\begin{align*}
		\left| \int_\Omega \vartheta(x) \partial_t u_j(t,x)\,\mathrm{d}x \right| & \le C \|\vartheta\|_{W_{n+1}^1} \left( 1 + Y_j[u_j(t)]^{1/2} \right)
	\end{align*}
	and a duality argument ensures that
	\begin{equation*}
		\|\partial_t u_j(t)\|_{(W_{n+1}^1)'} \le C \left( 1 + Y_j[u_j(t)]^{1/2} \right)\,.
	\end{equation*}
	Combining~\eqref{d6b} and the above inequality gives
	\begin{subequations}\label{d70}
	\begin{equation}
		\int_0^t \|\partial_t u_j(s)\|_{(W_{n+1}^1)'}^2\,\mathrm{d}s \le C (1+t)\,. \label{d7a}
	\end{equation}
	We next employ a similar argument to deduce from~\eqref{d60} and~\eqref{d101} that
	\begin{equation}
		\int_0^t \|\partial_t (u_j+v_j)(s)\|_{(W_{n+1}^1)'}^2\,\mathrm{d}s \le C (1+t)\,. \label{d7b}
	\end{equation}
	\end{subequations}
	Lemma~\ref{lem4.2} now readily follows from~\eqref{d70}.
\end{proof}

\subsection{Convergence}\label{sec4.2}

Having at hand Lemma~\ref{lem4.1} and Lemma~\ref{lem4.2}, we are in a position to establish compactness properties of the sequences $(u_j)_{j\ge 1}$ and $(v_j)_{j\ge 1}$ and begin with a simple consequence of~\eqref{d6a}, \eqref{d7}, the compactness of the embedding of $L_2(\Omega)$ in $W_{n+1}^1(\Omega)'$, and \cite[Corollary~1]{Sim1987}.

\begin{lemma}\label{lem4.3}
	The sequences $(u_j)_{j\ge 1}$ and $(v_j)_{j\ge 1}$ are relatively compact in $C([0,T],W_{n+1}^1(\Omega)')$ for any~$T>0$.
\end{lemma}

Though allowing us to pass to the limit in the linear terms involved in~\eqref{ujNLGS}, Lemma~\ref{lem4.3} does not carry enough information to pass to the limit in the nonlinear terms in~\eqref{ujNLGS}. Thus, strong compactness is required to handle these terms. Despite the lack of gradient estimates here, it was observed in \cite[Theorem~4]{BBM2001} and \cite[Theorem~1.2]{Pon2004} that, given a bounded sequence $(z_j)_{j\ge 1}$ in $L_2(\Omega)$, estimates on $(Y_j(z_j))_{j\ge 1}$ imply the relative compactness of $(z_j)_{j\ge 1}$ in $L_2(\Omega)$. Time-dependent extensions of~\cite[Theorem~4]{BBM2001} and \cite[Theorem~1.2]{Pon2004} are developed in~\cite[Appendix~B]{ElSk2023} and~\cite[Appendix~B]{CES2023} when $\Omega=\mathbb{T}^n$ is the $n$-dimensional torus, and we prove a similar result in Proposition~\ref{propA.1} below.

\begin{lemma}\label{lem4.4}
	For any $T>0$, the sequences $(u_j)_{j\ge 1}$ and $(v_j)_{j\ge 1}$ are relatively compact in $L_2((0,T)\times \Omega)$ and their cluster points belong to $L_2((0,T),H^1(\Omega))$.
\end{lemma}

\begin{proof}
	Recalling 
	\begin{equation*}
		m_2 = \int_{\mathbb{R}^n} |x|^2 \varphi(x)\,\mathrm{d}x > 0\,,
	\end{equation*}
	see~\eqref{m2}, we define
	\begin{equation}
		\varrho_j(x) := j^{n+2} \frac{|x|^2}{m_2} \varphi(jx) = j^n \varrho_1(jx)\,, \qquad x\in\mathbb{R}^n\,, \quad j\ge 1\,. \label{d20}
	\end{equation}
	Since $\varphi\in C_0^\infty(\mathbb{R}^n)$ is non-negative and radially symmetric, the function $\varrho_1$ satisfies the assumptions of Proposition~\ref{propA.1}. Moreover, 
	\begin{equation*}
		(u_j)_{j\ge 1} \;\;\text{ is bounded in }\;\; L_2((0,T)\times\Omega) \cap L_2((0,T),W_{n+1}^1(\Omega)') 
	\end{equation*}
	by Lemma~\ref{lem4.1} and Lemma~\ref{lem4.2}, while we infer from~\eqref{d6b} that
	\begin{align*}
		& \int_0^T \int_{\Omega\times\Omega} \frac{\big| u_j(t,x)-u_j(t,y) \big|^2}{|x-y|^2} \varrho_j(x-y)\,\mathrm{d}(x,y)\mathrm{d}t \\
		& \qquad = \frac{1}{m_2}  \int_0^T \int_{\Omega\times\Omega} \chi_{j}(x,y) \big| u_j(t,x)-u_j(t,y) \big|^2\,\mathrm{d}(x,y)\mathrm{d}t \\
		& \qquad = \frac{1}{m_2} \int_0^T Y_j[u_j(t)]\,\mathrm{d}t \le \frac{C_1}{m_2} (1+T)\,.
	\end{align*} 
	Consequently, the sequence $(u_j)_{j\ge 1}$ satisfies the assumptions of Proposition~\ref{propA.1} (with $p=r=2$ and $q=n+1$) and an application of Proposition~\ref{propA.1} provides the relative compactness of $(u_j)_{j\ge 1}$ in $L_2((0,T)\times\Omega)$, together with the $L_2((0,T),H^1(\Omega))$-regularity of its cluster points. A similar argument allows us to handle the sequence $(v_j)_{j\ge 1}$ and thereby complete the proof of Lemma~\ref{lem4.4}.
\end{proof}

\begin{proof}[Proof of Theorem~\ref{thm1.3}]
	Owing to~\eqref{d3}, \eqref{d6a}, Lemma~\ref{lem4.3} and Lemma~\ref{lem4.4}, a diagonal process ensures the existence of a sequence $(j_k)_{k\ge 1}$, $j_k\to\infty$, and non-negative functions
	\begin{align*}
		u & \in C([0,\infty),W_{n+1}^1(\Omega)')\cap L_\infty((0,\infty)\times\Omega)\cap L_{2,\mathrm{loc}}([0,\infty),H^1(\Omega))\,, \\
		 v & \in C([0,\infty),W_{n+1}^1(\Omega)')\cap L_\infty((0,\infty),L_2(\Omega)) \cap L_{2,\mathrm{loc}}([0,\infty),H^1(\Omega))\,,
	\end{align*}
	such that, for any $t>0$, 
	\begin{align}
		\lim_{k\to\infty} \sup_{s\in [0,t]} \left[ \|(u_{j_k} - u)(s) \|_{(W_{n+1}^1)'} + \|(v_{j_k} - v)(s) \|_{(W_{n+1}^1)'} \right] & = 0\,, \label{d21}\\
		\lim_{k\to\infty} \int_0^t \left[ \|(u_{j_k} - u)(s) \|_2^2 + \|(v_{j_k} - v)(s) \|_2^2 \right]\,\mathrm{d}s & = 0\,, \label{d22} 
	\end{align}
	and
	\begin{equation}
		\lim_{k\to\infty} \left[ |(u_{j_k} - u)(s,x)| + |(v_{j_k} - v)(s,x)| \right] = 0 \;\;\text{ for a.e. }\;\; (s,x)\in (0,t)\times\Omega\,. \label{d23}
	\end{equation}
	
	We first use the previous convergences to identify the limit of the nonlinear terms in~\eqref{CjNLGS}. For $t>0$ and $k\ge 1$, 
	\begin{equation}
		\begin{split}
			\int_0^t \int_\Omega \left| \left( u_{j_k} v_{j_k}^2 - uv^2 \right)(s,x) \right| \,\mathrm{d}x\mathrm{d}s & = \int_0^t \int_\Omega u_{j_k} (s,x) \left| \left( v_{j_k}^2 - v^2 \right)(s,x) \right| \,\mathrm{d}x\mathrm{d}s \\ 
			& \qquad + \int_0^t \int_\Omega v^2(s,x) \left| \left( u_{j_k} - u \right)(s,x) \right| \,\mathrm{d}x\mathrm{d}s\,.
		\end{split} \label{d24}
	\end{equation}
	On the one hand, it follows from~\eqref{d3}, \eqref{d6a}, \eqref{d22}, and H\"older's inequality that
	\begin{align*}
		& \int_0^t \int_\Omega u_{j_k} (s,x) \left| \left( v_{j_k}^2 - v^2 \right)(s,x) \right| \,\mathrm{d}x\mathrm{d}s \\
		& \qquad \le \int_0^t \|u_{j_k}(s)\|_\infty \|(v_{j_k}+v)(s)\|_2 \|(v_{j_k}-v)(s)\|_2\,\mathrm{d}s \\
		& \qquad \le C_0 \left( \int_0^t \|(v_{j_k}+v)(s)\|_2^2\,\mathrm{d}s \right)^{1/2} \left( \int_0^t \|(v_{j_k}-v)(s)\|_2^2\,\mathrm{d}s \right)^{1/2} \\
		& \qquad \le 2 C_0 \sqrt{C_1} \left( \int_0^t \|(v_{j_k}-v)(s)\|_2^2\,\mathrm{d}s \right)^{1/2}\,,
	\end{align*}
	so that, by~\eqref{d22},
	\begin{equation}
		\lim_{k\to\infty} \int_0^t \int_\Omega u_{j_k} (s,x) \left| \left( v_{j_k}^2 - v^2 \right)(s,x) \right| \,\mathrm{d}x\mathrm{d}s = 0\,. \label{d25}
	\end{equation}
	On the other hand, since $v^2\in L_1((0,t)\times\Omega)$, we deduce from~\eqref{d3}, \eqref{d23}, and Lebesgue's dominated convergence theorem that
	\begin{equation}
		\lim_{k\to\infty} \int_0^t \int_\Omega v^2(s,x) \left| \left( u_{j_k} - u \right)(s,x) \right| \,\mathrm{d}x\mathrm{d}s = 0\,. \label{d26}
	\end{equation}
	Gathering~\eqref{d24}, \eqref{d25}, and~\eqref{d26}, we conclude that
	\begin{equation}
		\lim_{k\to\infty} \int_0^t \int_\Omega \left| \left( u_{j_k} v_{j_k}^2 - uv^2 \right)(s,x) \right| \,\mathrm{d}x\mathrm{d}s = 0\,. \label{d27}
	\end{equation}
	
	Now, it follows from~\eqref{d22}, \eqref{d27}, and Proposition~\ref{propB.1} that $u=U$, where $U$ is the unique mild solution to 
	\begin{equation*}
		\begin{split}
			\partial_t U & = \frac{m_2 d_1}{2n} \Delta U - uv^2 + f(1-u) \;\;\text{ in }\;\; (0,\infty)\times\Omega\,, \\
			\nabla U\cdot \mathbf{n} & = 0 \;\;\text{ on }\;\; (0,\infty)\times\partial\Omega\,, \\
			U(0) & = u^0 \;\;\text{ in }\;\; \Omega\,;
		\end{split}
	\end{equation*}
that is, $u$ is the mild solution to
		\begin{equation*}
		\begin{split}
			\partial_t u & = \frac{m_2 d_1}{2n} \Delta u - uv^2 + f(1-u) \;\;\text{ in }\;\; (0,\infty)\times\Omega\,, \\
			\nabla u\cdot \mathbf{n} & = 0 \;\;\text{ on }\;\; (0,\infty)\times\partial\Omega\,, \\
			u(0) & = u^0 \;\;\text{ in }\;\; \Omega\,.
		\end{split}
	\end{equation*}
	A similar argument ensures that $v$ is the unique mild solution to
	\begin{equation*}
		\begin{split}
			\partial_t v & = \frac{m_2 d_2}{2n} \Delta v + uv^2 - (f+\kappa)v \;\;\text{ in }\;\; (0,\infty)\times\Omega\,, \\
			\nabla v\cdot \mathbf{n} & = 0 \;\;\text{ on }\;\; (0,\infty)\times\partial\Omega\,, \\
			v(0) & = v^0 \;\;\text{ in }\;\; \Omega\,.
		\end{split}
	\end{equation*}
	Recalling the already obtained regularity on $(u,v)$ allows us to deduce that $u$ and $v$ satisfy~\eqref{p5} and~\eqref{p6}, respectively, thereby completing the proof.
\end{proof}

\section{Extension to the Dirichlet setting}\label{sec5}

The approach used to establish the well-posedness of~\eqref{iNLGS} is actually quite flexible and can be easily extended to other nonlocal operators such as a nonlocal version of homogeneous Dirichlet boundary conditions, which differs from~\eqref{nlo}. Specifically, in this section only, given a measurable function $\gamma: \mathbb{R}^n\times\mathbb{R}^n\to [0,\infty)$ and $z\in L_\infty(\Omega)$, we define the nonlocal operator $\Gamma_\gamma$ by
\begin{equation}
\Gamma_\gamma z(x) := \int_{\mathbb{R}^n} \gamma(x,y) \big(\bar{z}(y)-z(x)\big)\,\mathrm{d}y \,, \qquad x\in\Omega\,, \label{nloD}
\end{equation}
where $\bar{z}(x)=z(x)$ for $x\in\Omega$ and $\bar{z}(x)=0$ for $x\in \Omega^c = \mathbb{R}^n\setminus\Omega$. With this definition, we consider the problem
\begin{subequations}\label{CDLGS}
	\begin{align}
		\partial_t u & = d_1 \Gamma_{\gamma_1} u - u v^2 + f (1-u) \;\;\text{ in }\;\; (0,\infty)\times\Omega\,,  \label{uNLGSx} \\
		\partial_t v & = d_2 \Gamma_{\gamma_2} v + u v^2- (f + \kappa) v \;\;\text{ in }\;\; (0,\infty)\times\Omega\,, \label{vNLGSx} \\
u&=v=0  \;\;\text{ in }\;\; (0,\infty)\times\Omega^c\,,  \label{uNLGSxx} \\
		(u,v)(0) & = \big( u^0,v^0) \;\;\text{ in }\;\; \Omega\,,  \label{iNLGSx}
	\end{align}
\end{subequations} 
and report the following well-posedness result.

\begin{theorem}[Well-posedness]\label{thm5.2.1}
	Assume that $\gamma_\ell: \mathbb{R}^n\times\mathbb{R}^n\to [0,\infty)$ is a measurable function satisfying 
	\begin{equation}
	\int_{\Omega} \gamma_\ell(y,x)\,\mathrm{d} y\le \int_{\mathbb{R}^n} \gamma_\ell(x,y)\,\mathrm{d} y \le  \gamma_\infty <\infty\,, \qquad x\in\Omega\,, \label{gammaD}
	\end{equation}
	for $\ell\in\{1,2\}$ and consider $(u^0,v^0)\in X^+\times X^+$. Then there is a unique non-negative global solution 
	\begin{equation*}
		(u,v)\in C^1\big([0,\infty),X^+\times X^+\big)
	\end{equation*}
	to~\eqref{CDLGS} which is bounded; that is, 
	\begin{equation*}
		\sup_{t\ge 0} \big\{ \|u(t)\|_{\infty} + \|v(t)\|_{\infty} \big\} < \infty\,.
	\end{equation*}
	In fact, the mapping $(u^0,v^0)\mapsto (u,v)$ defines a global semiflow on $X^+\times X^+$.
	
	Assume further that $\gamma_\ell\in C(\bar{\Omega}\times \bar{\Omega})$ for $\ell\in\{1,2\}$. Then the same result is true when replacing $X$ by~$C(\bar{\Omega})$.
\end{theorem}

The proof is the same as that of Theorem~\ref{thm1.1}, to which we refer. The only change is in the analogue of Lemma~\ref{lem2.1}, where~\eqref{b4} is replaced by 
\begin{equation}
	e^{t\Gamma_{\gamma}}\ge 0 \;\;\text{ and }\;\; e^{t\Gamma_{\gamma}} 1\le 1\,, \qquad t\ge 0\,. \label{b4D}
\end{equation} 
The proof of~\eqref{b4D} follows the lines of that of~\eqref{b5} and relies on~\eqref{gammaD}. 


\appendix

\section{A compactness result}\label{secA}

Let $\Omega$ be a bounded domain of $\mathbb{R}^n$, $n\ge 2$, and consider a non-negative radially symmetric function $\varrho\in C_0^\infty(\mathbb{R}^n)$ such that \begin{equation}
	\int_{\mathbb{R}^n} \varrho(x)\,\mathrm{d}x = 1\,. \label{ap01}
\end{equation}
For $x\in\mathbb{R}^n$ and $j\ge 1$, we put $\varrho_j(x) := j^n \varrho(jx)$ and note that $\varrho_j\in C_0^\infty(\mathbb{R}^n)$ with $\|\varrho_j\|_{L_1(\mathbb{R}^n)} = 1$.

Next, for an open subset $\mathcal{O}$ of $\mathbb{R}^n$, $p\in (1,\infty)$, $j\ge 1$, and $g\in L_p(\mathcal{O})$, we define
\begin{equation}
	\Lambda_j(g,\mathcal{O}) := \int_{\mathcal{O}\times\mathcal{O}} \frac{|g(x)-g(y)|^p}{|x-y|^p} \varrho_j(x-y)\,\mathrm{d}(x,y) \in [0,\infty]\,, \label{ap02}
\end{equation}
and set $\Lambda_j(\cdot) := \Lambda_j(\cdot,\mathbb{R}^n)$ for simplicity.

\begin{proposition}\label{propA.1}
	Let $(p,q,r)\in (1,\infty)^3$, $T>0$, and consider a sequence $(f_j)_{j\ge 1}$ of measurable functions on $(0,T)\times\Omega$ satisfying
	\begin{subequations}\label{ap3-5}
	\begin{equation}
		f_j \in L_p((0,T)\times\Omega)\cap L_r((0,T),W_q^1(\Omega)')\,, \qquad j\ge 1\,, \label{ap00}
	\end{equation}
	\begin{equation}
		\int_0^T \Big[ \|f_j(t)\|_p^p + \Lambda_j(f_j(t),\Omega) \Big] \mathrm{d}t \le K_0\,, \qquad j\ge 1\,, \label{ap03}
	\end{equation}
	as well as
	\begin{equation}
		\int_0^T \|\partial_t f_j(t)\|_{(W_q^1)'}^r\,\mathrm{d}t \le K_0\,, \qquad j\ge 1\,, \label{ap04}
	\end{equation}
	\end{subequations}
	for some $K_0>0$. Then $(f_j)_{j\ge 1}$ is relatively compact in $L_p((0,T)\times\Omega)$ and its cluster points in $L_p((0,T)\times\Omega)$ belong to $L_p((0,T),W_p^1(\Omega))$.
\end{proposition}

As already mentioned, Proposition~\ref{propA.1} is similar to \cite[Appendix~B]{CES2023} and \cite[Appendix~B]{ElSk2023} and is in the spirit of \cite[Corollary~4]{Sim1987}. Its proof relies on the compactness results established in \cite[Theorem~4]{BBM2001} and \cite[Theorem~1.2]{Pon2004} for sequences of time-independent functions and requires several intermediate results. 

For further use, we fix a non-negative radially symmetric function $\Phi\in C_0^\infty(\mathbb{R}^n)$ satisfying
\begin{equation}
	\mathrm{supp}\,\Phi \subset B_1(0)\,, \qquad 0 \le \Phi \le 1\,, \qquad \int_{\mathbb{R}^n} \Phi(x)\,\mathrm{d}x =1\,, \label{ap05}
\end{equation}
and set $\Phi_\delta(x) := \delta^{-n} \Phi(x/\delta)$ for $x\in\mathbb{R}^n$ and $\delta\in (0,1)$. We begin with an estimate on $\Phi_\delta*g-g$ for $g\in L_p(\mathbb{R}^n)$ and argue as in the proof of \cite[Proposition~4.2]{Pon2004} to derive the following result.

\begin{lemma}\label{lemA.2}
	Let $\delta_0\in (0,1)$ and consider $j\ge 1$ such that
	\begin{equation}
		\int_{B_{\delta_0}(0)} \varrho_j(x)\,\mathrm{d}x \ge \frac{1}{2}\,. \label{ap06}
	\end{equation}
	There exists $K_1>0$ depending only on $n$ and $p$ such that, for any $g\in L_p(\mathbb{R}^n)$ satisfying
	\begin{equation}
		\Lambda_j(g) = \int_{\mathbb{R}^n\times\mathbb{R}^n} \frac{|g(x)-g(y)|^p}{|x-y|^p} \varrho_j(x-y)\,\mathrm{d}(x,y) < \infty \label{ap07}
	\end{equation}
	and any $\delta\in (0,\delta_0)$, there holds
	\begin{equation}
		\int_{\mathbb{R}^n} \left| (\Phi_\delta*g - g)(x) \right|^p\,\mathrm{d}x \le K_1 \delta_0^p \Lambda_j(g)\,. \label{ap08}
	\end{equation}
\end{lemma}

In order to use Lemma~\ref{lemA.2} with functions defined only on $\Omega$, we next study the behaviour of the functional $\Lambda_j(\cdot,\Omega)$ with respect to space truncation.

\begin{lemma}\label{lemA.3}
	Let $\psi\in C_0^\infty(\Omega)$ and consider $j\ge 1$ and $g\in L_p(\Omega)$ such that $\Lambda_j(g,\Omega)<\infty$. Then $g\psi \in L_p(\mathbb{R}^n)$ with
	\begin{align}
		\|g\psi\|_{L_p(\mathbb{R}^n)} & \le \|\psi\|_\infty \|g\|_p\,, \label{ap09} \\
		\Lambda_j(g\psi) & \le 2^p \left[ \|\psi\|_\infty \Lambda_j(g,\Omega) + \|\nabla\psi\|_\infty \|g\|_p^p \right]\,. \label{ap10}
	\end{align}
\end{lemma}

\begin{proof}
	We define $g\psi$ on $\mathbb{R}^n$ by setting $(g\psi)(x)=g(x)\psi(x)$ for $x\in\Omega$ and $(g\psi)(x)=0$ for $x\in\mathbb{R}^n\setminus\Omega$. Since~\eqref{ap09} is obvious, we turn to~\eqref{ap10} and deduce from the properties of $\psi$ and the symmetry of~$\varrho_j$ that
	\begin{align*}
		\Lambda_j(g\psi) & = \int_{\Omega\times\Omega} \frac{\big| (g\psi)(x) -(g\psi)(y) \big|^p}{|x-y|^p} \varrho_j(x-y)\,\mathrm{d}(x,y) \\
		& \qquad + 2 \int_{\Omega\times(\mathbb{R}^n\setminus\Omega)} \frac{\big| (g\psi)(x) \big|^p}{|x-y|^p} \varrho_j(x-y)\,\mathrm{d}(x,y) \\
		& \le 2^{p-1} \int_{\Omega\times\Omega} \big[ |\psi(x)|^p |g(x)-g(y)|^p + |g(y)|^p |\psi(x)-\psi(y)|^p \big] \frac{\varrho_j(x-y)}{|x-y|^p}\,\mathrm{d}(x,y) \\
		& \qquad + 2 \int_{\Omega\times(\mathbb{R}^n\setminus\Omega)} \frac{ |g(x)|^p |\psi(x)-\psi(y)|^p}{|x-y|^p} \varrho_j(x-y)\,\mathrm{d}(x,y) \\
		& \le 2^{p-1} \|\psi\|_\infty \Lambda_j(g,\Omega) + 2^{p-1} \|\nabla\psi\|_\infty \int_{\Omega\times\Omega} |g(y)|^p \varrho_j(x-y)\,\mathrm{d}(x,y) \\
		& \qquad + 2 \|\nabla\psi\|_\infty \int_{\Omega\times(\mathbb{R}^n\setminus\Omega)} |g(x)|^p \varrho_j(x-y)\,\mathrm{d}(x,y)\\
		& \le 2^p \|\psi\|_\infty \Lambda_j(g,\Omega) + 2^p \|\nabla\psi\|_\infty \int_{\Omega} |g(x)|^p \int_{\mathbb{R}^n} \varrho_j(x-y)\,\mathrm{d}(x,y)\,,
	\end{align*}
	and we use~\eqref{ap01} to complete the proof.
\end{proof}

After this preparation, we may start to investigate the compactness issue and, as in \cite[Appendix~B]{CES2023} and \cite[Appendix~B]{ElSk2023}, we first establish the relative compactness of the sequence $\big( (\psi f_j)*\Phi_\delta \big)_{j\ge 1}$ for each $\delta\in (0,1)$ and $\psi\in C_0^\infty(\Omega)$. 

\begin{lemma}\label{lemA.4}
		Let $(p,q,r)\in (1,\infty)^3$, $T>0$, and a sequence $(f_j)_{j\ge 1}$ of measurable functions on $(0,T)\times\Omega$ satisfying~\eqref{ap3-5}. For any non-negative function $\psi\in C_0^\infty(\Omega)$ and $\delta\in (0,1)$, the sequence $\big( (\psi f_j)*\Phi_\delta \big)_{j\ge 1}$ is relatively compact in $L_p((0,T)\times\mathbb{R}^n)$.
\end{lemma}

\begin{proof}
	Introducing $U_1 := \{ x\in\mathbb{R}^n\ :\ d(x,\Omega)<1 \}$, we note that, for $t\in(0,T)$ and $x\not\in U_1$,
	\begin{align*}
		g_j(t,x) & := \big[ (\psi f_j(t,\cdot))*\Phi_\delta \big](x) = \int_{\mathbb{R}^n} \psi(y) f_j(t,y) \Phi_\delta(x-y)\,\mathrm{d}y \\
		& = \int_{\Omega} \psi(y) f_j(t,y) \Phi_\delta(x-y)\,\mathrm{d}y = 0\,,
	\end{align*}
	since $\mathrm{supp}\, \Phi_\delta \subset B_\delta(0) \subset B_1(0)$ by~\eqref{ap05}. Therefore,
	\begin{equation}
		\mathrm{supp}\, g_j(t) \subset U_1\,, \qquad t\in (0,T)\,, \ j\ge 1\,. \label{ap11}
	\end{equation}
	On the one hand, it follows from~\eqref{ap03} and the Riesz-Fr\'echet-Kolmogorov theorem, see \cite[Theorem~IV.25 \& Corollary~IV.27]{Bre2011} for instance, that
	\begin{equation}
		\left( x \mapsto \int_{t_1}^{t_2} g_j(t,x)\,\mathrm{d}t \right)_{j\ge 1} \;\;\text{ is relatively compact in }\;\; L_p(U_1) \label{ap12}
	\end{equation}
	for all $0<t_1<t_2<T$. On the other hand, for $h\in (0,T)$, $j\ge 1$, and $(t,x)\in (0,T-h)\times U_1$,
	\begin{align*}
		|g_j(t+h,x)-g_j(t,x)| & = \left| \int_{\mathbb{R}^n} \psi(y) [f_j(t+h,y)-f_j(t,y)] \Phi_\delta(x-y)\,\mathrm{d}y \right| \\
		& = \left| \int_t^{t+h} \big\langle \partial_t f_j(s) , \psi \Phi_\delta(x-\cdot) \big\rangle\,\mathrm{d}s \right| \\
		& \le \int_t^{t+h} \|\partial_t f_j(s)\|_{(W_q^1)'} \|\psi \Phi_\delta(x-\cdot)\|_{W_q^1}\,\mathrm{d}s \,.
	\end{align*}
	Owing to the properties of $\psi$ and $\Phi_\delta$,
	\begin{align*}
		\|\psi \Phi_\delta(x-\cdot)\|_{W_q^1} & = \|\psi \Phi_\delta(x-\cdot)\|_{W_q^1(\mathbb{R}^n)} \\
		& \le \|\psi\|_\infty \|\Phi_\delta(x-\cdot)\|_{W_q^1(\mathbb{R}^n)} + \|\nabla\psi\|_\infty \|\Phi_\delta(x-\cdot)\|_{L_q(\mathbb{R}^n)} \\
		& \le \|\psi\|_\infty \|\Phi_\delta\|_{W_q^1(\mathbb{R}^n)} + \|\nabla\psi\|_\infty \|\Phi_\delta\|_{L_q(\mathbb{R}^n)} \\
		& \le 2 \|\psi\|_{W_\infty^1} \|\Phi\|_{W_q^1(\mathbb{R}^n)} \delta^{-[q(n+1)-n]/q}\,.
	\end{align*} 
	Therefore,
	\begin{equation*}
		|g_j(t+h,x)-g_j(t,x)| \le 2 \|\psi\|_{W_\infty^1} \|\Phi\|_{W_q^1(\mathbb{R}^n)} \delta^{-[q(n+1)-n]/q} \int_t^{t+h} \|\partial_t f_j(s)\|_{(W_q^1)'}\,\mathrm{d}s \,,
	\end{equation*}
	and we infer from~\eqref{ap04} and H\"older's inequality that
	\begin{equation*}
		|g_j(t+h,x)-g_j(t,x)| \le 2 \|\psi\|_{W_\infty^1} \|\Phi\|_{W_q^1(\mathbb{R}^n)} h^{(r-1)/r} K_0^{1/r} \delta^{-[q(n+1)-n]/q}\,.
	\end{equation*}
	Integrating with respect to $(t,x)\in (0,T-h)\times U_1$, we are led to
	\begin{align}
		& \int_0^{T-h} \int_{U_1} |g_j(t+h,x)-g_j(t,x)|\,\mathrm{d}x\mathrm{d}t \nonumber \\
		& \hspace{2cm} \le 2 T |U_1| \|\psi\|_{W_\infty^1} \|\Phi\|_{W_q^1(\mathbb{R}^n)} h^{(r-1)/r} K_0^{1/r} \delta^{-[q(n+1)-n]/q}\,. \label{ap13}
	\end{align}
	Owing to~\eqref{ap12} and~\eqref{ap13}, we are in a position to apply \cite[Theorem~1]{Sim1987} and conclude that $(g_j)_{j\ge 1}$ is relatively compact in $L_p((0,T)\times U_1)$ and thus in $L_p((0,T)\times\mathbb{R}^n)$ according to~\eqref{ap11}.
\end{proof}

The next step is the relative compactness of $(f_j)_{j\ge 1}$ in $L_{p,\mathrm{loc}}((0,T)\times\Omega)$.

\begin{lemma}\label{lemA.5}
	Let $(p,q,r)\in (1,\infty)^3$, $T>0$, and a sequence $(f_j)_{j\ge 1}$ of measurable functions on $(0,T)\times\Omega$ satisfying~\eqref{ap3-5}. For any non-negative function $\psi\in C_0^\infty(\Omega)$, the sequence $\big( \psi f_j \big)_{j\ge 1}$ is relatively compact in $L_p((0,T)\times\Omega)$.
\end{lemma}

\begin{proof}
	We argue as in the proof of \cite[Proposition~4.2]{Pon2004} and consider $\delta_0\in (0,1)$. The properties of the sequence $(\varrho_j)_{j\ge 1}$ entail that there is $J_{\delta_0}\ge 1$ such that
	\begin{equation*}
		\int_{B_{\delta_0}(0)} \varrho_j(x)\,\mathrm{d}x \ge \frac{1}{2}\,, \qquad j\ge J_{\delta_0}\,,
	\end{equation*}
	so that $\varrho_j$ satisfies~\eqref{ap06} for any $j\ge J_{\delta_0}$. Since $\psi f_j(t)$ belongs to $L_p(\mathbb{R}^n)$ for $j\ge 1$ and a.e. $t\in (0,T)$, we infer from~\eqref{ap03}, Lemma~\ref{lemA.2}, and Lemma~\ref{lemA.3} that, for $\delta\in (0,\delta_0)$ and $j\ge J_{\delta_0}$, 
	\begin{align*}
		& \int_0^T \int_{\mathbb{R}^n} \Big| \big[ (\psi f_j(t,\cdot))*\Phi_\delta \big](x) - \psi(x) f_j(t,x) \Big|^p\,\mathrm{d}x\mathrm{d}t \\
		& \qquad \le K_1 \delta_0^p \int_0^T \Lambda_j(\psi f_j(t))\,\mathrm{d}t \\
		& \qquad \le (2\delta_0)^p K_1 \left[ \|\psi\|_\infty \int_0^T \Lambda_j(f_j(t),\Omega)\,\mathrm{d}t + \|\nabla\psi\|_\infty \int_0^T \|f_j(t)\|_p^p\,\mathrm{d}t \right] \\
		& \qquad \le K_0 K_1 (2\delta_0)^p \|\psi\|_{W_\infty^1}\,.
	\end{align*}
	Recalling Lemma~\ref{lemA.4}, the above inequality implies that $(\psi f_j)_{j\ge 1}$ is arbitrarily close to relatively compact sequences in $L_p((0,T)\times\mathbb{R}^n)$, so that it is also relatively compact in $L_p((0,T)\times\mathbb{R}^n)$ and thus in $L_p((0,T)\times\Omega)$.
\end{proof}

In view of Lemma~\ref{lemA.5}, we are left to handle the behaviour of $(f_j)_{j\ge 1}$ near the boundary of $\Omega$. In this direction, we recall the following result established in \cite[Lemma~5.1]{Pon2004}. 

\begin{lemma}[{\cite[Lemma~5.1]{Pon2004}}] \label{lemA.6}
	There are $r_0>0$ depending on $\Omega$ and $\varrho$ and constants $(K_2,K_3)\in (0,\infty)^2$ depending on $n$, $\Omega$, and $p$ with the following property: given $r\in (0,r_0)$, there is $I_r\ge 1$ such that, for every $j\ge I_r$ and $g\in L_p(\Omega)$, 
	\begin{equation*}
		\int_\Omega |g(x)|^p\,\mathrm{d}x \le K_2 \int_{\Omega_r} |g(x)|^p\,\mathrm{d}x + K_3 r^p \Lambda_j(g,\Omega)\,,
	\end{equation*}
	where $\Omega_r := \{ x\in\Omega\ :\ d(x,\partial\Omega) >r \}$.
\end{lemma}

\begin{proof}[Proof of Proposition~\ref{propA.1}]
	We argue as in the proof of \cite[Theorem~1.2]{Pon2004} with the help of Lemma~\ref{lemA.6}. We first readily infer from~\eqref{ap03}, Lemma~\ref{lemA.5}, and a diagonal process that there are a subsequence $(j_k)_{k\ge 1}$, $j_k\to\infty$, and $f\in L_p((0,T)\times\Omega)$ such that
	\begin{equation}
		\begin{split}
			f_{j_k} & \rightharpoonup f \;\;\text{ in }\;\; L_p((0,T)\times\Omega)\,, \\
			f_{j_k} & \longrightarrow f \;\;\text{ in }\;\; L_p((0,T)\times\Omega_r)\,, 
		\end{split} \label{ap14}
	\end{equation}
	for each $r\in (0,r_0)$. Moreover, arguing as in the proof of \cite[Equation~(24)]{Pon2004}, we deduce from~\eqref{ap03} and~\eqref{ap14} that
	\begin{equation}
		f \in L_p((0,T), W_p^1(\Omega))\,. \label{ap15}
	\end{equation}
	Now, let $r\in (0,r_0)$ and $j\ge j_r$. It follows from~\eqref{ap03}, \eqref{ap15},  Lemma~\ref{lemA.6} (with $g=f_{j_k}-f$), and \cite[Theorem~1]{BBM2001} that
	\begin{align*}
		\int_0^T \int_\Omega \big| f_{j_k}(t,x) - f(t,x) \big|^p\,\mathrm{d}x\mathrm{d}t & \le K_2 \int_0^T \int_{\Omega_r} \big| f_{j_k}(t,x) - f(t,x) \big|^p\,\mathrm{d}x\mathrm{d}t \\
		& \qquad + K_3 r^p \int_0^T \Lambda_{j_k}\big( (f_{j_k}-f)(t),\Omega \big)\,\mathrm{d}t \\
		& \le K_2 \int_0^T \int_{\Omega_r} \big| f_{j_k}(t,x) - f(t,x) \big|^p\,\mathrm{d}x\mathrm{d}t \\
		& \qquad + 2^{p-1} K_3 r^p \int_0^T \Big[ \Lambda_{j_k}\big(f_{j_k}(t),\Omega \big) + \Lambda_{j_k}\big(f(t),\Omega \big) \Big]\,\mathrm{d}t \\
		& \le K_2 \int_0^T \int_{\Omega_r} \big| f_{j_k}(t,x) - f(t,x) \big|^p\,\mathrm{d}x\mathrm{d}t \\
		& \qquad + 2^{p-1} K_3 r^p \left[ K_0 + K_4 \|\varrho\|_{L_1(\mathbb{R}^n)} \int_0^T \|f(t)\|_{W_p^1}\,\mathrm{d}t \right]\,,
	\end{align*}
	where $K_4$ is a positive constant depending only on $\Omega$ and $p$. Recalling~\eqref{ap01}, we may let $k\to\infty$ in the above inequality and use~\eqref{ap14} to deduce that
	\begin{equation*}
		\limsup_{k\to\infty} \int_0^T \int_\Omega \big| f_{j_k}(t,x) - f(t,x) \big|^p\,\mathrm{d}x\mathrm{d}t \le 2^{p-1} K_3 r^p \left[ K_0 + K_4 \int_0^T \|f(t)\|_{W_p^1}\,\mathrm{d}t \right]
	\end{equation*}
	for all $r\in (0,r_0)$. We then take the limit $r\to 0$ in the above inequality to complete the proof.
\end{proof}

\section{Diffusive limit in $L_1$: an auxiliary result}\label{secB}

We recall that, for $j\ge 1$, the kernel~$\chi_j$ and the nonlocal operator~$\Gamma_{\chi_j}$ are defined by~\eqref{d1000} and~\eqref{d100}, respectively.

\begin{proposition}\label{propB.1}
	Let $\Omega$ be a bounded domain of $\mathbb{R}^n$ with $C^{2+\alpha}$-smooth boundary $\partial\Omega$, $T>0$, and $d>0$. We consider a sequence $(S_j)_{j\ge 1}$ in $L_\infty((0,T)\times\Omega)$ and a sequence $(w_j^0)_{j\ge 1}$ in $L_\infty(\Omega)$ such that 
	\begin{equation}
		\lim_{j\to\infty} \big[ \|S_j-S\|_{L_1((0,T)\times\Omega)} + \|w_j^0-w^0\|_1 \big] = 0 \label{bp01}
	\end{equation}
	for some $S\in L_1((0,T)\times\Omega)$ and $w^0\in L_1(\Omega)$. For $j\ge 1$, let $w_j\in C^1([0,\infty),L_\infty(\Omega))$ be the solution to
	\begin{equation}
		\begin{split}
			\partial_t w_j & = d \Gamma_{\chi_j}w_j + S_j \;\;\text{ in }\;\; (0,T)\times\Omega\,, \\
			w_j(0) & = w_j^0 \;\;\text{ in }\;\; \Omega\,,
		\end{split}\label{bp02}
	\end{equation}
	where $\chi_j$ and $\Gamma_{\chi_j}$ are defined in~\eqref{d1000} and~\eqref{d100}, respectively. If there is $w_\infty\in L_1((0,T)\times\Omega)$ such that
	\begin{equation}
		\lim_{j\to\infty} \|w_j-w_\infty\|_{L_1((0,T)\times\Omega)} = 0\,, \label{bp03}
	\end{equation}
	then $w_\infty=w$, where $w$ denotes the unique mild solution to 
	\begin{equation}
		\begin{split}
			\partial_t w & = D \Delta w + S \;\;\text{ in }\;\; (0,T)\times\Omega\,, \\
			\nabla w\cdot \mathbf{n} & = 0 \;\;\text{ on }\;\; (0,T)\times\partial\Omega\,, \\
			w(0) & = w^0 \;\;\text{ in }\;\; \Omega\,,
		\end{split}\label{bp04}
	\end{equation}
	with $\mathbf{n}$  denoting  the outward normal unit vector field to $\partial\Omega$ and
	\begin{equation*}
		D := \frac{d}{2n} \int_{\mathbb{R}^n} |z|^2 \varphi(z)\,\mathrm{d}z\,.
	\end{equation*}
\end{proposition}

\begin{proof}
	The proof is inspired by \cite[Sections~3.2.2 \&~3.2.4]{AVMRTM2010}. We consider $\Sigma\in C_0^\infty((0,T)\times\Omega)$, $W^0\in C^{2+\alpha}(\bar{\Omega})$ and let $W\in C^{1+\alpha/2,2+\alpha}([0,T]\times\bar{\Omega})$ be the classical solution to 
	\begin{equation}
		\begin{split}
			\partial_t W & = D \Delta W + \Sigma \;\;\text{ in }\;\; (0,T)\times\Omega\,, \\
			\nabla W\cdot \mathbf{n} & = 0 \;\;\text{ on }\;\; (0,T)\times\partial\Omega\,, \\
			W(0) & = W^0 \;\;\text{ in }\;\; \Omega\,.
		\end{split}\label{bp05}
	\end{equation}
	
	On the one hand, it follows from~\eqref{bp04}, \eqref{bp05}, and the contraction property of the heat semigroup in $L_1(\Omega)$ that
	\begin{equation}
		\sup_{t\in [0,T]} \|W(t)-w(t)\|_1 \le \|W^0-w^0\|_1 + \|S-\Sigma\|_{L_1((0,T)\times\Omega)}\,. \label{bp06}
	\end{equation}
	
	On the other hand, it follows from~\eqref{bp02} and~\eqref{bp05} that $w_j-W$ solves
		\begin{equation}
		\begin{split}
			\partial_t (w_j-W) & = d\Gamma_{\chi_j}(w_j-W) + F_j + S_j - \Sigma\;\;\text{ in }\;\; (0,T)\times\Omega\,, \\
			(w_j-W)(0) & = w_j^0 - W^0 \;\;\text{ in }\;\; \Omega\,,
		\end{split}\label{bp07}
	\end{equation}
	with $F_j := d \Gamma_{\chi_j}W - D \Delta W$. Since
	\begin{align*}
		\int_\Omega \Gamma_{\chi_j}z(x) \mathrm{sign}(z(x))\,\mathrm{d}x & = \int_{\Omega\times\Omega} \chi_j(x,y) z(y) \mathrm{sign}(z(x))\,\mathrm{d}(x,y) \\
		& \qquad - \int_{\Omega\times\Omega} \chi_j(x,y) |z(x)|\,\mathrm{d}(x,y) \\
		& \le \int_{\Omega\times\Omega} \chi_j(x,y) |z(y)|\,\mathrm{d}(x,y) \\
		& \qquad - \int_{\Omega\times\Omega} \chi_j(y,x) |z(y)|\,\mathrm{d}(x,y) \\ 
		& \le 0
	\end{align*}
	for all $z\in L_1(\Omega)$ due to the symmetry of $\chi_j$, we infer from~\eqref{bp07} that
	\begin{align}
		\sup_{t\in [0,T]} \|w_j(t)-W(t)\|_1 & \le \|w_j^0 - W^0\|_1 + \|F_j + S_j - \Sigma\|_{L_1((0,T)\times\Omega)} \nonumber\\
		& \le \|w_j^0 - w^0\|_1 + \|w^0 - W^0\|_1 + \|F_j\|_{L_1((0,T)\times\Omega)} \nonumber\\
		& \qquad + \|S_j - S\|_{L_1((0,T)\times\Omega)} + \|S - \Sigma\|_{L_1((0,T)\times\Omega)}\,. \label{bp08}
	\end{align}
	
	We now estimate $F_j$. To this end, we proceed as in \cite[Section~3.2.2]{AVMRTM2010} and denote an extension of $W$ to $[0,T]\times\mathbb{R}^n$ by $\mathcal{W}\in C^{1+\alpha/2,2+\alpha}([0,T]\times\mathbb{R}^n)$. Then $F_j = F_{j,1} + F_{j,2}$ with
	\begin{align*}
		F_{j,1}(t,x) & := d \int_{\mathbb{R}^n} \chi_j(x,y) \big[\mathcal{W}(t,y)-\mathcal{W}(t,x) \big]\,\mathrm{d}y - D \Delta \mathcal{W}(t,x)\,,  \\
		F_{j,2}(t,x) & := - d \int_{\mathbb{R}^n\setminus\Omega} \chi_j(x,y) \big[ \mathcal{W}(t,y) - \mathcal{W}(t,x) \big]\,\mathrm{d}y
	\end{align*}
	for $(t,x)\in [0,T]\times\bar{\Omega}$. We first argue as in the proof of \cite[Equation~(3.15)]{AVMRTM2010} and use a Taylor expansion and the radial symmetry of $\varphi$ to deduce that there is a positive constant $B_1(\mathcal{W})$ depending only on $n$, $\Omega$, $\varphi$, $T$, and $\mathcal{W}$ such that
	\begin{equation}
		\|F_{j,1}\|_{L_1((0,T)\times\Omega)} \le B_1(\mathcal{W}) j^{-\alpha}\,. \label{bp09}
	\end{equation}
	Next, if $x\in\Omega_{1/j} := \{ y\in \Omega\ :\ \mathrm{dist}(y,\partial\Omega)<1/j\}$ and $j$ is large enough, then the orthogonal projection $Px$ of $x$ on $\partial\Omega$ is well-defined and it follows from \cite[Lemma~3.14]{AVMRTM2010} that there is a positive constant $B_2(\mathcal{W})$ depending only on $n$, $\Omega$, $\varphi$, $T$, and $\mathcal{W}$ such that, for $t\in [0,T]$,
	\begin{equation}
		\left| F_{j,2}(t,x) + \frac{d}{2} \int_{\mathbb{R}^n\setminus\Omega} \chi_j(x,y) \big\langle D^2\mathcal{W}(t,Px) (y-Px) , (x-Px) \big\rangle\,\mathrm{d}y \right| \le B_2(\mathcal{W}) j^{-\alpha}\,. \label{bp10}
	\end{equation}
	Now, $|x-Px|\le 1/j$, 
	\begin{align*}
		& \mathrm{supp}\,\chi_j(x,\cdot)\subset \mathcal{K} := \{y\in \mathbb{R}^n\ :\ \mathrm{dist}(y,\Omega)\le 1\}\,, \\
		& |y-Px| \le |y-x| + |x-Px| \le 2/j \;\;\text{ for }\;\; y\in\mathrm{supp}\,\chi_j(x,\cdot)\,,
	\end{align*}
	so that
	\begin{align*}
		& \left| \frac{d}{2} \int_{\mathbb{R}^n\setminus\Omega} \chi_j(x,y) \big\langle D^2\mathcal{W}(t,Px) (y-Px) , (x-Px) \big\rangle\,\mathrm{d}y \right| \\ 
		& \qquad \le \frac{d \|D^2\mathcal{W}(t)\|_{L_\infty(\mathcal{K})}}{j^2} \int_{\mathbb{R}^n\setminus\Omega} \chi_j(x,y)\,\mathrm{d}y \\
		& \qquad \le j^n d \|D^2\mathcal{W}\|_{L_\infty([0,T]\times\mathcal{K})} \int_{\mathbb{R}^n} \varphi(j(x-y))\,\mathrm{d}y \\
		& \qquad \le d \|D^2\mathcal{W}\|_{L_\infty([0,T]\times\mathcal{K})} \,.
	\end{align*}
	Recalling~\eqref{bp10}, we have shown that there is a positive constant $B_3(\mathcal{W})$ depending only on $n$, $\Omega$, $\varphi$, $T$, and $\mathcal{W}$ such that
	\begin{equation}
		F_{j,2}(t,x) \le B_3(\mathcal{W})\,, \qquad (t,x)\in [0,T]\times \Omega_{1/j}\,. \label{bp11}
	\end{equation}
	Finally, if $x\in\Omega\setminus\Omega_{1/j}$, then $|x-y|\ge 1/j$ for all $y\in\mathbb{R}^n\setminus\Omega$. Thus, due to \eqref{d1000},
	\begin{equation}
		F_{j,2}(t,x) =0\,, \qquad (t,x)\in [0,T]\times \big( \Omega\setminus\Omega_{1/j} \big)\,. \label{bp12}
	\end{equation}
	Collecting~\eqref{bp06}, \eqref{bp08}, \eqref{bp09}, \eqref{bp11}, and~\eqref{bp12}, we find
	\begin{align*}
		\sup_{t\in [0,T]} \|w_j(t)-w(t)\|_1 & \le 2 \|W^0-w^0\|_1 + 2 \|S-\Sigma\|_{L_1((0,T)\times\Omega)} + \|w_j^0 - w^0\|_1 \\
		& \qquad + B_1(\mathcal{W}) j^{-\alpha} + T B_3(\mathcal{W}) |\Omega_{1/j}| + \|S_j - S\|_{L_1((0,T)\times\Omega)} \,. 
	\end{align*}
	Hence,
	\begin{align*}
		\|w_j-w\|_{L_1((0,T)\times\Omega)} & \le 2T \|W^0-w^0\|_1 + 2T \|S-\Sigma\|_{L_1((0,T)\times\Omega)} + T \|w_j^0 - w^0\|_1 \\
		& \qquad + T B_1(\mathcal{W}) j^{-\alpha} + T^2 B_3(\mathcal{W}) |\Omega_{1/j}| + T \|S_j - S\|_{L_1((0,T)\times\Omega)} \,. 
	\end{align*}
	We then let $j\to\infty$ in the above inequality and deduce from~\eqref{bp01}, \eqref{bp03}, and the definition of $\Omega_{1/j}$ that
	\begin{equation*}
		\|w_\infty-w\|_{L_1((0,T)\times\Omega)} \le 2T \|W^0-w^0\|_1 + 2T \|S-\Sigma\|_{L_1((0,T)\times\Omega)} \,.
	\end{equation*}
	Since the above inequality is valid for all $(\Sigma,W^0)\in C_0^\infty((0,T)\times \Omega)\times C_0^\infty(\Omega)$, we now use the density of $C_0^\infty((0,T)\times \Omega)$ in $L_1((0,T)\times\Omega)$, along with that of $C_0^\infty(\Omega)$ in $L_1(\Omega)$, to complete the proof. 
\end{proof}

\bibliographystyle{siam}
\bibliography{GrayScott}

\begin{thebibliography}{10}

\bibitem{AbHu2023}
{\sc H.~Abels and C.~Hurm}, {\em Strong nonlocal-to-local convergence of the
  {C}ahn-{H}illiard equation and its operator}, 2023.

\bibitem{AbTe2022}
{\sc H.~Abels and Y.~Terasawa}, {\em Convergence of a nonlocal to a local
  diffuse interface model for two-phase flow with unmatched densities},
  Discrete Contin. Dyn. Syst., Ser. S, 15 (2022), pp.~1871--1881.

\bibitem{Aman1990}
{\sc H.~Amann}, {\em Ordinary differential equations. {An} introduction to
  nonlinear analysis. {Transl}. from the {German} by {Gerhard} {Metzen}},
  vol.~13 of De Gruyter Stud. Math., Berlin: Walter de Gruyter, 1990.

\bibitem{AVMRTM2010}
{\sc F.~Andreu-Vaillo, J.~M. Maz{\'o}n, J.~D. Rossi, and J.~J. Toledo-Melero},
  {\em Nonlocal diffusion problems}, vol.~165 of Math. Surv. Monogr.,
  Providence, RI: American Mathematical Society (AMS); Madrid: Real Sociedad
  Matem{\'a}tica Espa{\~n}ola, 2010.

\bibitem{BBM2001}
{\sc J.~Bourgain, H.~Brezis, and P.~Mironescu}, {\em Another look at {S}obolev
  spaces}, in Optimal control and partial differential equations. In honour of
  Professor Alain Bensoussan's 60th birthday. Proceedings of the conference,
  Paris, France, December 4, 2000, Amsterdam: IOS Press; Tokyo: Ohmsha, 2001,
  pp.~439--455.

\bibitem{Bre2011}
{\sc H.~Brezis}, {\em Functional analysis, {Sobolev} spaces and partial
  differential equations}, Universitext, New York, NY: Springer, 2011.

\bibitem{CJW2022}
{\sc L.~Cappanera, G.~Jaramillo, and C.~Ward}, {\em Analysis and simulations of
  a nonlocal {G}ray-{S}cott model}, 2022.

\bibitem{CES2023}
{\sc J.~A. Carrillo, C.~Elbar, and J.~Skrzeczkowski}, {\em Degenerate
  {C}ahn-{H}illiard systems: From nonlocal to local}, 2023.

\bibitem{Cas2017}
{\sc R.~Castelli}, {\em Rigorous computation of non-uniform patterns for the
  2-dimensional {Gray}-{Scott} reaction-diffusion equation}, Acta Appl. Math.,
  151 (2017), pp.~27--52.

\bibitem{ChWa2011}
{\sc W.~Chen and M.~J. Ward}, {\em The stability and dynamics of localized spot
  patterns in the two-dimensional {Gray}-{Scott} model}, SIAM J. Appl. Dyn.
  Syst., 10 (2011), pp.~582--666.

\bibitem{DRST2020}
{\sc E.~Davoli, H.~Ranetbauer, L.~Scarpa, and L.~Trussardi}, {\em Degenerate
  nonlocal {Cahn}-{Hilliard} equations: well-posedness, regularity and local
  asymptotics}, Ann. Inst. Henri Poincar{\'e}, Anal. Non Lin{\'e}aire, 37
  (2020), pp.~627--651.

\bibitem{DST2021}
{\sc E.~Davoli, L.~Scarpa, and L.~Trussardi}, {\em Nonlocal-to-local
  convergence of {Cahn}-{Hilliard} equations: {Neumann} boundary conditions and
  viscosity terms}, Arch. Ration. Mech. Anal., 239 (2021), pp.~117--149.

\bibitem{DKZ1997}
{\sc A.~Doelman, T.~J. Kaper, and P.~A. Zegeling}, {\em Pattern formation in
  the one-dimensional {Gray}-{Scott} model}, Nonlinearity, 10 (1997),
  pp.~523--563.

\bibitem{ElSk2023}
{\sc C.~Elbar and J.~Skrzeczkowski}, {\em Degenerate {C}ahn-{H}illiard
  equation: from nonlocal to local}, J. Differ. Equations, 364 (2023),
  pp.~576--611.

\bibitem{GZK2018}
{\sc P.~Gandhi, Y.~R. Zelnik, and E.~Knobloch}, {\em Spatially localized
  structures in the {Gray}-{Scott} model}, Philos. Trans. R. Soc. Lond., A,
  Math. Phys. Eng. Sci., 376 (2018), p.~20.
\newblock Id/No 20170375.

\bibitem{HPT2000}
{\sc J.~K. Hale, L.~A. Peletier, and W.~C. Troy}, {\em Exact homoclinic and
  heteroclinic solutions of the {G}ray-{S}cott model for autocatalysis}, SIAM
  J. Appl. Math., 61 (2000), pp.~102--130.

\bibitem{HMP1987}
{\sc S.~L. Hollis, R.~H. Martin, and M.~Pierre}, {\em Global existence and
  boundedness in reaction-diffusion systems}, SIAM J. Math. Anal., 18 (1987),
  pp.~744--761.

\bibitem{KWW2006}
{\sc T.~Kolokolnikov, M.~J. Ward, and J.~Wei}, {\em Zigzag and breakup
  instabilities of stripes and rings in the two-dimensional {Gray}-{Scott}
  model}, Stud. Appl. Math., 116 (2006), pp.~35--95.

\bibitem{MGR2004}
{\sc J.~S. McGough and K.~Riley}, {\em Pattern formation in the {G}ray-{S}cott
  model}, Nonlinear Anal., Real World Appl., 5 (2004), pp.~105--121.

\bibitem{MRST2019}
{\sc S.~Melchionna, H.~Ranetbauer, L.~Scarpa, and L.~Trussardi}, {\em From
  nonlocal to local {C}ahn-{H}illiard equation}, Adv. Math. Sci. Appl., 28
  (2019), pp.~197--211.

\bibitem{MoKa2004}
{\sc D.~S. Morgan and T.~J. Kaper}, {\em Axisymmetric ring solutions of the 2d
  {Gray}-{Scott} model and their destabilization into spots}, Physica D, 192
  (2004), pp.~33--62.

\bibitem{Mou2020}
{\sc A.~Moussa}, {\em From nonlocal to classical
  {S}higesada-{K}awasaki-{T}eramoto systems: triangular case with bounded
  coefficients}, SIAM J. Math. Anal., 52 (2020), pp.~42--64.

\bibitem{Pea1993}
{\sc J.~E. Pearson}, {\em Complex patterns in a simple system}, Science, 261
  (1993), pp.~189--192.

\bibitem{PeWa2009}
{\sc R.~Peng and M.~X. Wang}, {\em Some nonexistence results for nonconstant
  stationary solutions to the {Gray}-{Scott} model in a bounded domain}, Appl.
  Math. Lett., 22 (2009), pp.~569--573.

\bibitem{Pon2004}
{\sc A.~C. Ponce}, {\em An estimate in the spirit of {P}oincar{\'e}'s
  inequality}, J. Eur. Math. Soc. (JEMS), 6 (2004), pp.~1--15.

\bibitem{Sim1987}
{\sc J.~Simon}, {\em Compact sets in the space {{\(L^ p(0,T;B)\)}}}, Ann. Mat.
  Pura Appl. (4), 146 (1987), pp.~65--96.

\bibitem{WSWK2019}
{\sc T.~Wang, F.~Song, H.~Wang, and G.~E. Karniadakis}, {\em Fractional
  {Gray}-{Scott} model: well-posedness, discretization, and simulations},
  Comput. Methods Appl. Mech. Eng., 347 (2019), pp.~1030--1049.

\bibitem{You2008}
{\sc Y.~You}, {\em Global attractors of the {G}ray-{S}cott equations}, Commun.
  Pure Appl. Anal., 7 (2008), pp.~947--970.

\end{thebibliography}

\end{document}